\documentclass[11pt]{article}
\usepackage[a4paper, left=2cm, right=2cm, top=3cm, bottom=3cm]{geometry}
\usepackage[comma, round]{natbib}
\usepackage{float}
\usepackage{amssymb}
\usepackage{amsthm}
\usepackage{amsmath}
\usepackage{titletoc}
\usepackage{mathrsfs}
\usepackage[colorlinks, allcolors=blue, linktocpage]{hyperref}
\usepackage{amsfonts}
\usepackage{graphicx}
\usepackage[linesnumbered, ruled, vlined]{algorithm2e}
\usepackage{esint}
\usepackage{bbm}
\usepackage{bm}
\usepackage{mathtools}
\usepackage[shortlabels, inline]{enumitem}
\usepackage{authblk}
\usepackage{todonotes}
\usepackage{color}

\numberwithin{equation}{section}
\numberwithin{figure}{section}

\theoremstyle{plain}
\newtheorem{thm}{Theorem}[section]
\newtheorem{cor}[thm]{Corollary}
\newtheorem{lem}[thm]{Lemma}
\newtheorem{prop}[thm]{Proposition}

\theoremstyle{definition}
\newtheorem{defn}[thm]{Definition}
\newtheorem{asmp}[thm]{Assumption}
\newtheorem{rmk}[thm]{Remark}
\newtheorem{exam}[thm]{Example}

\newcommand{\cB}{\mathcal{B}}

\newcommand{\cF}{\mathcal{F}}

\newcommand{\cH}{\mathcal{H}}

\newcommand{\cL}{\mathcal{L}}

\newcommand{\cN}{\mathcal{N}}

\newcommand{\cR}{\mathcal{R}}

\newcommand{\cX}{\mathcal{X}}
\newcommand{\cY}{\mathcal{Y}}

\newcommand{\bF}{\mathbf{F}} 
 
\newcommand{\bH}{\mathbf{H}}

\newcommand{\bP}{\mathbf{P}}

\newcommand{\bbN}{\mathbb{N}}

\newcommand{\bbR}{\mathbb{R}}

\newcommand{\bbZ}{\mathbb{Z}}

\newcommand{\scrP}{\mathscr{P}}

\newcommand{\scrU}{\mathscr{U}}

\newcommand{\frB}{\mathfrak{B}}
\newcommand{\frC}{\mathfrak{C}}
\newcommand{\frD}{\mathfrak{D}}

\newcommand{\frN}{\mathfrak{N}}

\newcommand{\1}{\mathbbm{1}}
\newcommand{\AW}{\mathcal{AW}}
\newcommand{\W}{\mathcal{W}}

\newcommand{\la}{\langle}
\newcommand{\ra}{\rangle}

\newcommand{\md}{\mathop{}\mathopen\mathrm{d}}
\newcommand{\me}{\mathop{}\mathopen\mathrm{e}}
\newcommand{\E}{E}
\renewcommand{\P}{P}

\newcommand{\Id}{\operatorname{Id}}

\newcommand{\Law}{\operatorname{Law}}
\newcommand{\supp}{\mathrm{supp}}

\renewcommand{\c}{\mathfrak{c}}
\newcommand{\bc}{\mathfrak{bc}}

\newcommand{\tr}{\operatorname{tr}}
\newcommand{\HS}{\mathrm{HS}}

\newcommand{\sync}{\mathrm{sync}}
\newcommand{\diag}{\operatorname{diag}}
\newcommand{\range}{\operatorname{range}}

\def\lb{\mathopen{}\mathclose\bgroup\left}
\def\rb{\aftergroup\egroup\right}

\title{A transfer principle for computing the adapted Wasserstein distance between stochastic processes}
\author{Yifan Jiang\thanks{Email: {\tt yifan.jiang@maths.ox.ac.uk}.} }
\author{Fang Rui Lim\thanks{Email: {\tt fang.lim@maths.ox.ac.uk}.}}
\affil{Mathematical Institute, University of Oxford}
\date{}

\begin{document}
\newgeometry{top=2cm, bottom=2cm, right=2cm, left=2cm}

\maketitle

\begin{abstract}
	We propose a transfer principle to study the adapted 2-Wasserstein distance between stochastic processes.
	First, we obtain an explicit formula for the distance between real-valued mean-square continuous Gaussian processes by introducing the causal factorization as an infinite-dimensional analogue of the Cholesky decomposition for operators on Hilbert spaces.
	We discuss the existence and uniqueness of this causal factorization and link it to the canonical representation of Gaussian processes.
	As a byproduct, we characterize mean-square continuous Gaussian Volterra processes in terms of their natural filtrations.
	Moreover, for real-valued fractional stochastic differential equations, we show that the synchronous coupling between the driving fractional noises attains the adapted Wasserstein distance under some monotonicity conditions.
	Our results cover a wide class of stochastic processes which are neither Markov processes nor semi-martingales, including fractional Brownian motions and fractional Ornstein--Uhlenbeck processes.



	\medskip
\end{abstract}

\tableofcontents
\restoregeometry
\clearpage
\newpage

\section{Introduction}
Stochastic processes, the building block of stochastic analysis, can be viewed as path-valued random variables.
From this perspective, the convergence of stochastic processes can naturally be induced by the weak convergence of their laws as probability measures on the path space.
However, this `static' viewpoint turns out to be insufficient for `dynamic' problems, especially for many key applications in mathematical finance and beyond.
In particular, the value of a stochastic optimal stopping problem is not continuous with respect to this weak topology \citep{backhoff-veraguas20Adapted,backhoff-veraguas22Adapted}.
Different notions of adapted topologies have been proposed to refine the weak topology, such as Aldous's extended weak topology \citep{aldous81Weak}, Hellwig's information topology \citep{hellwig96Sequential},  Hoover--Keisler topology \citep{hoover84Adapteda,hoover87Characterization}, nested distance \citep{pflug12Distance}, and the adapted Wasserstein distance \citep{lassalle18Causal,bion-nadal19Wassersteintype}.
In the seminal paper \citet{backhoff-veraguas20All}, these notions are unified and proven to be all equivalent to the initial topology of the optimal stopping problems in a discrete-time setting.
The essence of all aforementioned adapted topologies is to consider not only the law but also the conditional law of the stochastic process with respect to its natural filtration.
Or, in other words, to incorporate the information flow carried by the underlying process.

We focus on the adapted Wasserstein distance which was first introduced in \citet{lassalle18Causal} as a dynamic counterpart of the Wasserstein distance for stochastic processes.
For two stochastic processes \(X_{1}\) and \(X_{2}\) on a given probability space \((\Omega,\cF,\P)\), their adapted 2-Wasserstein distance is given by
\begin{equation}
	\AW_{2}(X_{1},X_{2}):=\inf_{\pi\in\Pi_{\bc}(X_{1},X_{2})}\E_{\pi}[\|X_{1}-X_{2}\|^{2}]^{1/2},
\end{equation}
where \(\|\cdot\|\) is the \(L^{2}\) norm on the path space, and  \(\Pi_{\bc}(X_{1},X_{2}):=\{\pi \text{ is bicausal}:\pi(\cdot\times \Omega)=\pi(\Omega\times\cdot)=\P\}\) is a subset of  couplings with an additional \emph{causality} constraint.
Heuristically speaking, under such constraint, \(\cF_{T}^{X_{1}}\) (the future of \(X_{1}\)) is conditionally independent of  \(\cF_{t}^{X_{2}}\) (the past of \(X_{2}\)) given \(\cF_{t}^{X_{1}}\) (the past of \(X_{1}\)), and vice versa.
We refer to Section \ref{sec-pre} for a more precise definition.
It has been applied to the analysis of various aspects of robust finance such as stability \citep{backhoff-veraguas20Adapted}, sensitivity \citep{bartl22Sensitivity,jiang24Sensitivity}, and model risk \citep{jiang24Duality,sauldubois24First}.
However, computing the adapted Wasserstein distance analytically, or even numerically, is difficult, due to the additional causality constraint.
Even in discrete time, few explicit formulas have been obtained for the adapted Wasserstein distance, see \citet{gunasingam25Adapteda,acciaio24Entropica,backhoff-veraguas17Causal}, etc.
In continuous time,  to the best of our knowledge,  there has been little to no results beyond the semi-martingale framework, see \citet{lassalle18Causal,bion-nadal19Wassersteintype,backhoff-veraguas22Adapted}, etc.

In this paper, we leverage a simple yet effective transfer principle to compute the explicit adapted Wasserstein distance between Gaussian processes and identify the optimal coupling between fractional stochastic differential equations.
Given a transport map \(T_{i}\) such that \(X_{i}=T_{i}(Y_{i})\) and \(X_{i}\), \(Y_{i}\) generate the same natural filtration, then \(\Pi_{\bc}(X_{1},X_{2})=\Pi_{\bc}(Y_{1},Y_{2})\) and
\begin{equation}
	\label{eqn-transfer}
	\AW_{2}(X_{1},X_{2})=\inf_{\Pi_{\bc}(Y_{1},Y_{2})} \E_{\pi}[\|T_{1}(Y_{1})-T_{2}(Y_{2})\|^{2}]^{1/2}.
\end{equation}
This principle transfers the original transport problem from \(X_{1}\) and \(X_{2}\) to \(Y_{1}\) and \(Y_{2}\) which could have a much simpler structure.
In particular, if \((Y_{i}(t))_{t\in I}\) has  independent marginals, then under any bicausal coupling, one can verify
\begin{equation}
	\label{eqn-margin}
	Y_{1}(t) \text{ is independent from } Y_{2}(s) \text{ for distinct } s,t\in I.
\end{equation}
In \citet{backhoff-veraguas22Adapted}, this principle has been already applied to transfer bicausal couplings between SDEs to bicausal couplings between Brownian motions.
To illustrate the idea, we consider a simpler example of discrete-time Gaussian processes from \citet{gunasingam25Adapteda}.
Let \(X_{i}\sim \cN(0,\Sigma_{i} )\) be an \(N\)-step 1D non-degenerate Gaussian process.
We construct \(X_{i}=K_{i}Y_{i}\) where  \(K_{i}\) is the Cholesky decomposition of \(\Sigma_{i}\) and \(Y_{i}\sim \cN(0,\Id_{N})\) is a standard Gaussian.
Indeed, \(X_{i}\) and \(Y_{i}\) generate the same natural filtration as \(K_{i}\) is lower triangular and invertible.
By applying the transfer principle and \eqref{eqn-margin}, we can calculate \(\AW_{2}(X_{1},X_{2})\) as
\begin{align*}
	AW_{2}(X_{1},X_{2})^{2} & =\tr(\Sigma_{1}+\Sigma_{2})-2\sup_{\pi\in\Pi_{\bc}(X_{1},X_{2})}\E_{\pi}\lb[\la X_{1},X_{2}\ra\rb]                               \\
	                        & = \tr(\Sigma_{1}+\Sigma_{2})-2\sup_{\pi\in\Pi_{\bc}(Y_{1},Y_{2})}\E_{\pi}\lb[\la K_{1}Y_{1},K_{2}Y_{2}\ra\rb]                    \\
	                        & = \tr(\Sigma_{1}+\Sigma_{2})-2\sup_{\pi\in\Pi_{\bc}(Y_{1},Y_{2})}\sum_{n=1}^{N}(K_{1}^{*}K_{2})_{n,n}\E_{\pi}[Y_{1}(n)Y_{2}(n)].
\end{align*}
This gives us \(\AW_{2}(X_{1},X_{2})^{2}=\tr(\Sigma_{1}+\Sigma_{2})-2\sum_{n=1}^{N}|(K_{1}^{*}K_{2})_{n,n}|\) as choosing \(\E_{\pi}[Y_{1}(n)Y_{2}(n)]\) to match the sign of the diagonal element \((K_{1}^{*}K_{2})_{n,n}\) attains the supremum.
Heuristically, we can view \(Y_{i}\) as a `nicer' coordinate system which leads to a `nicer' parameterization of the set of bicausal couplings, and hence simplifies the computation.

Our first main result extends the above example to a continuous-time setting and computes the adapted Wasserstein distance between  mean-square continuous Gaussian processes.
To apply the transfer principle, in Section \ref{sec-factorization}, we introduce a notion of `canonical causal factorization' as an infinite-dimensional analogue of the Cholesky decomposition for operators on Hilbert spaces.
This notion naturally bridges an algebraic object `nest algebra' \citep{davidson88Nest} and a probabilistic object `canonical representation'  \citep{hida60Canonical} of Gaussian processes.
Our results give an explicit formula of the adapted Wasserstein distance in terms of the canonical causal factorization of the covariance operator, or equivalently, of the canonical representation of the Gaussian process, see Theorem \ref{thm-gp}.
For example, any fractional Brownian motion \(B_{H}\) has a Molchan--Golosov representation given by
\(B_{H}(t)=\int_{0}^{t}k_{H}(t,s)\md B(s)\), where \(H\in(0,1)\) is the Hurst parameter, \(k_{H}\) is the Molchan--Golosov kernel \citep{molchan69Gaussian, decreusefond99Stochastic}
\begin{equation*}
	k_{H}(t,s)=\Gamma(H+1/2)^{-1}(t-s)^{H-1/2} F(H-1/2,1/2-H,H+1/2,1-t/s)\1_{\{s\leq t\}},
\end{equation*}
\(F(a,b,c,z)\) is the Gaussian hypergeometric function \(
F(a,b,c,z)=\sum_{n=0}^{\infty}\frac{(a)_{n}(b)_{n}}{(c)_{n}}\frac{z^{n}}{n!}\),
and \((x)_{n}=\Gamma(x+n)/\Gamma(x)\) is the Pochhammer symbol.
We have the following result as a direct application of Theorem~\ref{thm-gp}.
\begin{thm}
	\label{thm-fb}
	Let \(B_{H_{i}}\) be the fractional Brownian motion with Hurst parameter \(H_{i}\in(0,1)\).
	Then the adapted 2-Wasserstein distance bewteen \(B_{H_1}\) and \(B_{H_{2}}\) is given by
	\begin{equation*}
		\AW_{2}(B_{H_{1}},B_{H_{2}})^{2}=\int_{0}^{T}\int_{0}^{T}(k_{H_{1}}(t,s)-k_{H_{2}}(t,s))^{2}\md t\md s.
	\end{equation*}
	Moreover, the optimal coupling is given by the synchronous coupling between \(B_{H_{1}}\) and \(B_{H_{2}}\), i.e., they are driven by the same Brownian motion in their Molchan--Golosov representations.
\end{thm}
Our second result considers the adapted Wasserstein distances between fractional stochastic differential equations.
By applying the transfer principle, we reformulate the adapted Wasserstein distance as a stochastic optimal control problem of fractional SDEs.
The control only appears as the correlation between the driving noises.
We show the optimality of the synchronous coupling by adapting the path-dependent HJB equation framework from \citet{viens19martingale}, see Theorem~\ref{thm-fsde}.
In particular, for SDEs driven by fractional Brownian motions, we have the following result.
\begin{thm}
	Let \(X_{i}\) be the solution of the following fractional SDE
	\begin{equation*}
		X_{i}(t)=x_{i}+\int_{0}^{t}b_{i}(X_{i}(t))\md t + \int_{0}^{t}\sigma_{i}(X_{i}(t))\md B_{H_{i}}(t),
	\end{equation*}
	where \(B_{H_{i}}\) is the fractional Brownian motion with Hurst parameter \(H_{i}\in(1/2,1)\).
	We assume that \(b_{i},\sigma_{i}\in C^{2}\) with bounded first and second derivatives, and \(b_{i}'',\sigma_{i}''\) are uniformly continuous.
	Moreover, \(\sigma_{i}\) is positive, bounded, and bounded away from zero.
	Then, \(\AW_{2}(X_{1},X_{2})\) is attained by the synchronous coupling between \(B_{H_{1}}\) and \(B_{H_{2}}\).
\end{thm}

Admittedly, the regularity constraint in the above result is not optimal, as is often the case in classical stochastic control theory, where strong assumptions are imposed to ensure the verification theorem.
In a forthcoming work, we aim to relax the regularity constraint through a time-discretization approximation in the spirit of \citet{backhoff-veraguas22Adapted} and extend results to stochastic Volterra equations with monotone kernels.

To the best of our knowledge, this is the first work to investigate the adapted Wasserstein distance between fractional processes.
We stress that these processes are neither semimartingales nor Markovian,
which  precludes a direct application of techniques from the existing literature.
Their ability to capture long-range dependence and rough path behavior has led to impactful applications, notably in finance \citep{baillie96Long, rogers97Arbitrage, cont05Long}, in physics \citep{metzler00random}, in engineering \citep{levy-vehel05Fractals}, and filtering theory \citep{decreusefond98Fractional}.

\subsection{Related literature}
We review the existing literature on the computation of adapted Wasserstein distances.
For broader literature related causal optimal transport problems, we refer readers to \citet{backhoff-veraguas20All,bartl24Wasserstein,bartl25Wasserstein} and references therein.
In discrete time, \citet{gunasingam25Adapteda} computed explicitly the adapted Wasserstein distance between two 1D Gaussian processes.
More recently, \citet{acciaio24Entropica} extended the previous result to multi-dimensional Gaussian processes and also considered an entropic regularization. Both of these results leveraged a dynamic programming principle from \citet{backhoff-veraguas17Causal}, which is distinct from the transfer principle considered in this work.
Instead of computing the explicit formula of the adapted Wasserstein distance, the Knothe--Rosenblatt coupling is identified as the optimal coupling between  co-monotone distributions in discrete time \citep{ruschendorf85Wasserstein,backhoff-veraguas17Causal}.
In continuous time, it is shown in \citet{lassalle18Causal}, for a Cameron--Martin cost, the adapted Wasserstein distance between an arbitrary probability measure and the Wiener measure is equal to the square-root of its relative entropy with respect to the Wiener measure.
For \(L^{2}\) cost, it is shown in \citet{bion-nadal19Wassersteintype} and later in \citet{backhoff-veraguas22Adapted,robinson24Bicausal} that the synchronous coupling is the optimal coupling between two 1D SDEs.

Another line of research is to numerically compute the adapted Wasserstein distance by approximation or regularization.
These results are mainly in a discrete-time setting.
For instance, \citet{eckstein24Computational} proposed numerical algorithms to compute the entropic regularized adapted Wasserstein distance.
In  \citet{pflug16Empirical,backhoff22Estimating,acciaio24Convergence}, the authors studied various smoothed adapted empirical measures and derived the convergence rate to their limit under the adapted Wasserstein distance.

The notion of causality underpinning the adapted Wasserstein distance, when placed in the context of linear transformations between (finite-dimensional) vector spaces, naturally corresponds to the triangularity of these transformations. A suitable generalization of these triangular forms to Hilbert spaces is the nest algebra, which originates from the work of \citet{ringrose65Algebras}.
The nest algebra is a prime example of non-selfadjoint algebras and reflexive algebras in the sense of \citet{arveson74Operator}.
Early research focused on the structure of compact operators in nest algebras, see \citet{ringrose62SuperDiagonal,erdos68Operators}, etc.
Further developments include the characterization of the radical \citep{ringrose65Algebras}, unitary invariants \citep{erdos67Unitary}, and similarity invariants \citep{larson85Nest}.
The causal factorization introduced in Section~\ref{sec-factorization} is motivated by several pioneering works \citep{pitts88Factorization,anoussis97Factorisation,anoussis98Factorisation}.
We refer interested readers to \citet{davidson88Nest} for a more complete reference.


In order to study the prediction theory of Gaussian processes,  \citet{levy56Special} introduced the canonical representation of a Gaussian process, which provides a full description of its natural filtration. This canonical representation and the related notion of multiplicity was systematically investigated in \citet{hida60Canonical,hida93Gaussian,hitsuda68Representation} and extended by  \citet{cramer71Structural}  to general stochastic processes. In the sequel, we clarify the connection between the canonical representation of Gaussian processes and the (canonical) causal factorization of their covariance operators.

\subsection{Outline}
The rest of the paper is organized as follows.
In Section \ref{sec-pre}, we recall basic definitions and properties of the adapted Wasserstein distance and the canonical representation of Gaussian processes.
In Section~\ref{sec-factorization}, we introduce the (canonical) causal factorization and discuss its existence and uniqueness.
A characterization of the Gaussian Volterra processes is given in Corollary \ref{cor-volterra} which we believe is of independent interest.
In Section \ref{sec-Gaussian}, we apply the transfer principle to compute the adapted 2-Wasserstein distance between Gaussian processes.
An explicit formula for the distance is given in Theorems \ref{thm-gp} and~\ref{thm-gp2} for the unit multiplicity case and the higher multiplicity case respectively.
An optimal coupling is identified in both cases.
In Theorem \ref{thm-mart}, we consider the best martingale approximation to a fractional Brownian motion with respect to the adapted 2-Wasserstein distance.
In Section \ref{sec-fsde}, we study the adapted Wasserstein distance between fractional SDEs via a stochastic control reformulation.
We establish a verification theorem for additive fractional SDEs and reduce the multiplicative case into the additive case  via a Lamperti transform \citep{lamperti64simple}.
Some technical estimates are postponed to Section \ref{sec-estimate}.

\section{Preliminaries}
\label{sec-pre}
\subsection{Notations}

For a  Polish space \(\cX\), we equip it with its Borel \(\sigma\)-algebra \(\cB(\cX)\).
Let \(\scrP(\cX)\) be the space of Borel probability measures on \(\cX\) equipped with its weak topology.
Given \(\mu\in\scrP(\cX)\) and a \(\sigma\)-algebra \(\cF\subseteq\cB(\cX)\), we denote the completion of \(\cF\) under \(\mu\)  by \(\prescript{\mu}{}\cF\).

Let \(\cX\) and \(\cY\) be two Polish spaces.
Given \(\mu\in\scrP(\cX)\) and \(\nu\in\scrP(\cY)\),  the set of couplings between \(\mu\) and \(\nu\) is defined as
\begin{equation*}
	\Pi(\mu,\nu):=\{\pi\in\scrP(\cX\times\cY):\pi(\cdot\times \cY)=\mu(\cdot)\text{ and } \pi(\cX\times\cdot)=\nu(\cdot)\}.
\end{equation*}
Given \(\mu\in\scrP(\cX)\) and a measurable map \(\Phi:\cX\to\cY\), we define the \emph{pushforward} map \(\Phi_{\#}:\scrP(\cX)\to\scrP(\cY)\) by
\begin{equation*}
	\Phi_{\#}\mu:=\mu\circ \Phi^{-1} \quad\text{for any } \mu\in\scrP(\cX).
\end{equation*}
For any \(\pi\in \scrP(\cX\times \cY)\), we write
\begin{equation*}
	\pi(\md x,\md y)=\pi(\md x) \theta(x,\md y),
\end{equation*}
where  \(\theta\) is the Borel regular disintegration kernel.

Let \(\mu,\mu_1,\mu_2\) be positive measures on \([0,T]\).
We write \(H_{\mu}=L^{2}([0,T],\mu;\bbR)\) and \(\la \cdot , \cdot \ra_{\mu}\) as the inner product on \(H_{\mu}\) with the induced norm \(\|\cdot\|_{\mu}\).
We write \(H_{\mu,t}=\{f\in H_{\mu}: \supp(f)\subseteq[0,t]\}\) as a closed subspace of \(H_{\mu}\).
We equip \(H_{\mu}\) with the Borel \(\sigma\)-algebra \(\cB(H_{\mu})\) and its natural filtration \(\bH_{\mu}=(\cH_{\mu,t})_{t\in[0,T]}\), where \(\cH_{\mu,t}:=\sigma(f\in H_{\mu,t})\).
Here, we identify \(H_{\mu,t}\) with its dual \(H_{\mu,t}^{*}\).
By \(\frB(H_{\mu_{1}},H_{\mu_{2}})\) we denote the set of bounded linear operator \(A:H_{\mu_{1}}\to H_{\mu_{2}}\), and we write \(\frB(H_{\mu})=\frB(H_{\mu},H_{\mu})\).
Given a closed subspace \(N\subseteq H_{\mu}\), we denote the orthogonal projection onto \(N\) by \(P_{N}\).

We say an operator \(A\in\frB(H_{\mu})\) is positive if \(\la Af,f\ra_{\mu}\geq0\) for any \(f\in H_{\mu}\).
We say an operator \(A\in\frB(H_{\mu})\)  is trace-class, if \(\|A\|_{\tr}:=\sum_{k\geq1}\la |A|e_{k},e_{k}\ra_{\mu}\) is finite for an orthonormal basis \((e_{k})_{k\geq1}\) of \(H_{\mu}\).
An operator \(K:H_{\mu_{1}}\to H_{\mu_{2}}\) is Hilbert--Schmidt, if \(KK^{*}\) is trace-class where \(K^{*}\) is the dual operator of \(K\).
Its Hilbert--Schmidt norm is defined as \(\|K\|_{\HS}=\sqrt{\tr(KK^{*})}\).
We denote the set of Hilbert--Schmidt operators from \(H_{\mu_{1}}\) to \(H_{\mu_{2}}\) by \(\frB_{2}(H_{\mu_{1}},H_{\mu_{2}})\).
There exists an isometry from \(\frB_{2}(H_{\mu_{1}},H_{\mu_{2}})\)  to  \(L^{2}([0,T]^{2}, \mu_{1}\otimes \mu_{2};\bbR)\).
In fact, every Hilbert--Schmidt operator \(K\) has a kernel \(k\in L^{2}([0,T]^{2},\mu_{1}\otimes \mu_{2};\bbR)\) such that \(Kf(t)=\int_{0}^{T} k(t,s)f(s)\mu_{1}(\md s)\in H_{\mu_{2}}\).
We omit the subscript if \(\mu=\lambda\) the Lebesgue measure on \([0,T]\).

By \(\mu_{1}\gg\mu_{2}\) we denote \(\mu_{2}\) is absolutely continuous with respect to \(\mu_{1}\) and write their Radon--Nikodym derivative as \( \frac{\md \mu_{2}}{\md \mu_{1}}\).
We denote the geometric mean of \(\mu_{1}\) and \(\mu_{2}\) by \[\sqrt{\mu_{1}\mu_{2}}(\md t):= \sqrt{\frac{\md \mu_{1}}{\md (\mu_{1} + \mu_{2})}\frac{\md \mu_{2}}{\md (\mu_{1} + \mu_{2})}} (t)(\mu_{1}+\mu_{2})(\md t).\]

Let \(C([0,T];\bbR)\) be the continuous path space.
For a functional \(f\) on \(C([0,T];\bbR)\), we say \(f\) is Fr\'echet differentiable at \(\omega\in C([0,T];\bbR)\) if there exists a linear functional \(\partial_{\omega}f(\omega) \in C([0,T];\bbR)^{*}\) such that for any \(\eta\in C([0,T];\bbR)\) it holds
\begin{equation*}
	f(\omega+\eta)-f(\omega)=\la \eta, \partial_{\omega}f(\omega)\ra +o(\|\eta\|).
\end{equation*}
We call \(\partial_{\omega}f\) the Fr\'echet derivative of \(f\).
Similarly, we define the second Fr\'echet derivative of \(f\) and denote it as \(\partial^{2}_{\omega}f\).
Given two linear functionals \(f,g\in C([0,T];\bbR)^{*}\), we denote their tensor product as a bilinear functional given by
\begin{equation*}
	\la (\eta_{1},\eta_{2}),f\otimes g\ra=\la \eta_{1},f\ra \la \eta_{2},g\ra,
\end{equation*}
for any \(\eta_{1},\eta_{2}\in C([0,T];\bbR)\).

\subsection{Adapted Wasserstein distance}
In the spirit of \citet{lassalle18Causal,acciaio20Causal}, we present a seemingly different, but equivalent, definition of the adapted Wasserstein distance. One can view our approach as a strong formulation and the framework of filtered process developed in \citet{bartl24Wasserstein,bartl25Wasserstein,pammer24note} as a weak formulation.

We begin with the \emph{causal} transport map between two filtered Polish spaces.
\begin{defn}
	Let \((\Omega_{1},(\cF_{1,t})_{t\in[0,T]})\) and \((\Omega_{2},(\cF_{2,t})_{t\in[0,T]})\)  be two filtered Polish spaces. We say \(T:\Omega_{1}\to\Omega_{2}\) is \emph{causal} if for any \(t\in[0,T]\) it holds \(T^{-1}(\cF_{2,t})\subseteq \cF_{1,t}\).
\end{defn}

The above motivates the definition of (bi)causal couplings between two stochastic processes by viewing the disintegration kernel as a randomized transport map.
From now on, we fix a Polish probability space \((\Omega,\cF,\P)\) and consider stochastic processes on it.
We denote the (completed) natural filtration of a stochastic process \(X\) by \(\bF^{X}=(\cF^{X}_{t})_{t\in[0,T]}\) where \(\cF^{X}_{t}=\prescript{\P}{}\sigma(X(s):s\leq t)\).

\begin{defn}[Causal coupling]
	Let \(X_{1}\) and \(X_{2}\) be two stochastic processes on \(\Omega\).
	We say a coupling \(\pi\in\Pi(\P,\P)\) is \emph{causal} from \(X_{1}\) to \(X_{2}\) if for any \(t\in [0,T]\) and \(U\in \cF^{X_{2}}_{t}\)
	\begin{equation}
		\label{def-caus}
		\Omega\ni \omega_{1}\mapsto \theta(\omega_{1},U)\in\bbR
	\end{equation}
	is \(\cF^{X_{1}}_{t}\)--measurable, where \(\pi(\md \omega_{1},\md \omega_{2})=\pi(\md \omega_{1})\theta(\omega_{1},\md \omega_{2})\).
	We say \(\pi\in\Pi(\P,\P)\) is bicausal if it is causal and \([(x_{1},x_{2})\mapsto (x_{2},x_{1})]_{\#}\pi\) is causal from \(X_{2}\) to \(X_{1}\).
	We denote the set of causal (bicausal) couplings from \(X_{1}\) to \(X_{2}\) by \(\Pi_{\c}(X_{1},X_{2})\) \((\Pi_{\bc}(X_{1},X_{2}))\).
\end{defn}
In particular, if a Monge map  \(T: (\Omega,\bF^{X_{1}}) \to (\Omega,\bF^{X_{2}})\) is causal and measure preserving, then \((\Id,T)_{\#}P\in \Pi_{\c}(X_{1},X_{2})\).
The (bi)causal transport maps are shown to be dense among (bi)causal transport couplings under different settings, see \citet{beiglbock25Denseness,cont24Causal}.
We notice that the causality condition here only depends on the filtration of the underlying probability space and can be easily extended to the case where the source and target probability space \((\Omega,\cF,P)\) are different.

\begin{defn}
	We say a process \(X\) is mean-square continuous if \(t\mapsto X(t)\in L^{2}(\Omega,P)\) is continuous.
	For two mean-square continuous stochastic processes \(X_{1}\) and \(X_{2}\), their adapted Wasserstein distance is defined as
	\begin{equation}
		\label{eqn-ad}
		\AW_{2}(X_{1},X_{2}):=\inf_{\pi\in\Pi_{\bc}(X_{1},X_{2})}\E_{\pi}[\|X_{1}-X_{2}\|^{2}]^{1/2}.
	\end{equation}
\end{defn}

\begin{rmk}
	\label{rmk-ad}
	Strictly speaking, the adapted Wasserstein distance defined here is a pseudo distance on the set of mean-square continuous stochastic processes on \(\Omega\).
	For any \(X_{1}\) and \(X_{2}\) with the same law in \(H=L^{2}([0,T];\bbR)\), we have \(\AW_{2}(X_{1},X_{2})=0\).
	More generally, if we equip \(H\) with its Borel \(\sigma\)-algebra and a natural filtration \(\bH=(\cH_{t})_{t\in[0,T]}\) where \(\cH_{t}=\sigma(h\in H^{*}:\supp(h)\in[0,t])\).
	Then, for any bicausal coupling \(\hat{\pi}\) between \((H,\cB(H),\bH,\Law(X_{1}))\) and  \((H,\cB(H),\bH,\Law(X_{2}))\) can be lifted to a bicausal coupling \(\pi\) between \(X_{1}\) and \(X_{2}\).
	The lifted coupling \(\pi\) is given by
	\[\pi(\md \omega_{1},\md \omega_{2})= \hat{\pi}(\md x_{1},\md x_{2} )\theta_{1}(x_{1},\md \omega_{1})\theta_{2}(x_{2},\md \omega_{2}),\]
	where \(\pi_{i}=(X_{i},\Id_{\Omega})_{\#}P\) and \(\theta_{i}\) is the disintegration kernel of \(\pi_{i}\) with  \(\pi_{i}(\md x_{i}, \md \omega_{i})=\pi_{i}(\md x_{i}) \theta_{i}(x_{i},\md \omega_{i})\).
	One can verify \(\pi\), \(\pi_{i}\) are all bicausal, and in particular,
	\((X_{1},X_{2})_{\#}\pi= \hat{\pi}\), see \citet[Lemma 3.4]{eckstein24Computational} for more details.
	Therefore, this definition induces a true distance on the \emph{law} of mean-square continuous stochastic processes, and it has the advantage of spotlighting the role of the underlying filtration.
\end{rmk}

In the recent work of \citet{bartl25Wasserstein}, the above definition is referred as the strict adapted Wasserstein distance as the induced topology is strictly stronger than the initial topology of optimal stopping problems.
A relaxed version of \eqref{eqn-ad} is proposed such that all adapted topologies are equivalent in continuous time.
Nevertheless, the current definition enjoys better analytic properties and can be viewed as natural extension from discrete time to continuous time.




\subsection{Canonical representation of Gaussian processes}
We say \(X:\Omega\times[0,T]\to\bbR\)  is a 1D Gaussian process if for any \(t_{1},\ldots,t_{n}\in[0,T]\), the random vector \((X(t_{1}),\ldots,X(t_{n}))\) is Gaussian.
In this paper, we focus on centered and mean-square continuous Gaussian processes, i.e.,
\begin{equation*}
	\E[X(t)]=0 \text{ for any } t\in[0,T], \text{ and } t\mapsto X(t)\in L^{2}(\Omega,P) \text{ is continuous.}
\end{equation*}
Notice that mean-square continuity of \(X\) implies \(X\) has path in \(H=L^{2}([0,T],\lambda;\bbR)\) almost surely.
Hence, \(X_{\#}P\) yields a Gaussian measure on \(H\) whose covariance operator \(\Sigma:H\to H\) is given by  \(\la \Sigma f,g\ra:=\E[\la f, X\ra\la g,X\ra]\).
Moreover, \(\Sigma\) has a unique continuous kernel \(R(t,s)=\E[X(t)X(s)]\) such that
\begin{equation*}
	\Sigma f(t)=\int_{0}^{T} f(s) R(t,s)\md s.
\end{equation*}

We say a Gaussian process \(X\) is \emph{deterministic} if the behavior of \(X\) is completely determined by its behavior in an infinitesimal time, i.e.,
\begin{equation*}
	\bigcap_{t>0}\overline{\mathrm{span}\{X(s):s\in[0,t]\}}= \overline{\mathrm{span}\{X(s):s\in[0,T]\}}\subseteq L^{2}(\Omega,P);
\end{equation*}
and \(X\) is \emph{purely nondeterministic} if the information of \(X\) must have entered as a new impulse at  some  definite time in the past, i.e.,
\begin{equation}
	\label{eqn-pnd}
	\bigcap_{t>0}\overline{\mathrm{span}\{X(s):s\in[0,t]\}}= 0.
\end{equation}
We shall not confuse a deterministic Gaussian process with a deterministic path-valued random variable which is supported on a single path.
For example, \(X(t)=t\xi\) where \(\xi\sim \cN(0,1)\) is a deterministic Gaussian process but not a deterministic random variable.
We remark that in Corollary \ref{cor-volterra}, we show that \eqref{eqn-pnd} is equivalent to the condition that \(\cF_{0+}^{X}=\bigcap_{t>0}\cF_{t}^{X}\) is trivial.



In what follows, we introduce the canonical representation of a Gaussian process which was initiated by \citet{levy56Special}, and systematically studied by \citet{hida60Canonical} and \citet{cramer71Structural}.
It states that a centered,  mean-square continuous, and purely nondeterministic Gaussian process is essentially driven by a countable number of `noises'.
Such a representation is canonical in the sense that the `noises' precisely generate the same natural filtration as the one of the Gaussian process.
We adapt \citet[Theorem 4.1]{hida93Gaussian} to our setting.
\begin{thm}
	\label{thm-canonical}
	Let \(X\) be a centered, mean-square continuous, and purely nondeterministic Gaussian process.
	Then there exists a number \(N\in \bbN \cup \{\infty\}\) uniquely determined by \(X\), which will be called the multiplicity of \(X\).
	The Gaussian process \(X\) has a canonical representation in the form of
	\begin{equation}
		\label{eqn-can}
		X(t)=\sum_{n=1}^{N}\int_{0}^{t} k^{n}(t,s) \md M^{n}(s),
	\end{equation}
	satisfying the following conditions:
	\begin{enumerate}[label=(\roman*)]
		\item  \(\{M^{n}\}_{n=1}^{N}\) are independent Gaussian martingales with independent increments,
		\item   \(\mu^{n}(t):=[M^{n}](t)\) is continuous, non-decreasing,  and \(\mu^{1}(\md t) \gg  \mu^{2}(\md t) \gg \cdots\),
		\item  \(t\mapsto k^{n}(t,\cdot)\in H_{\mu^{n}}\) is continuous and \(\supp(k^{n}(t,\cdot))\subseteq [0,t]\),
		\item \(\bF^{X}=\bF^{M}\) with \(M=(M_{1},\dots, M_{N})\).
	\end{enumerate}
\end{thm}
In general, it is not easy to find the canonical representation of a Gaussian process.
The following result from \citet[Theorem 4.4]{hida93Gaussian} gives a characterization of the canonical representation.
\begin{thm}
	\label{thm-can}
	Let \(X\) be a Gaussian process with a representation of the form of \eqref{eqn-can}.
	Then it is a canonical representation if and only if for any \(T'\in[0,T]\) and \(f^{n}\in H_{\mu^{n}}\),
	\begin{equation*}
		g(t)=\sum_{n=1}^{N}\int_{0}^{t} k^{n}(t,s)f^{n}(s)\mu^{n}(\md s)=0 \text{ for all } t\in[0,T']
	\end{equation*}
	implies \(f^{n}=0\) on \([0,T']\) for all \(n\).
\end{thm}

\section{Causal factorization}
\label{sec-factorization}
In this section, we introduce the \emph{causal} factorization as  an analogue of Cholesky decomposition for positive operators on infinite dimensional Hilbert space.
We first recall some basic properties of Cholesky decomposition.
For any positive definite matrix \(A\in \bbR^{N\times N}\), there exists a lower triangular matrix \(L\) such that \(A=LL^{*}\).
If \(A\) is nondegenerate,  such a decomposition \(L\) is unique up to a multiplication by a diagonal matrix \(D\) with diagonal entries being \(\{1,-1\}\).

It is clear that lower triangularity is not an intrinsic property, but depends on the choice of the basis.
From a geometric viewpoint, let \(V=\bbR^{N}\) and \(V_{n}=\mathop{\mathrm{span}}\{e_{1},\dots,e_{n}\}\), where \({\bf e}=\{e_{n}\}_{n=1}^{N}\) is an orthonormal basis of \(V\).
A map \(A:V\to V\) is lower triangular (with respect to \(\bf e\)) if and only if \(V_{n}^{\bot}\) is invariant under \(A\) for any \(1\leq n\leq N\).

We focus on the covariance operator \(\Sigma\) associated to a centered, mean-square continuous, and purely nondeterministic Gaussian process \(X\), and denote the set of such operators as \(\frC(H)\).
Notice \(\frC(H)\) is a proper subset of positive trace operators on \(H\) with continuous kernel.
\begin{defn}
	Let \(\Sigma\in \frC(H)\).
	For a positive continuous measure \(\mu\) on \([0,T]\) and Hilbert--Schmidt operator \(K:(H_{\mu},\bH_{\mu})\to (H,\bH)\), we say \((K,\mu)\) is a \emph{causal} factorization of \(\Sigma\) if \(K\)  is \emph{causal} and \(\Sigma=KK^{*}\).
	We say \((K,\mu)\) is a canonical causal factorization if further \(K_{t}\) is injective for any \(t\in[0,T]\), where \(K_{t}:=P_{H_{t}}K|_{H_{\mu,t}}\).
\end{defn}

The following property justifies the causality condition as a natural extension to the lower traingularity in continuous time context.
\begin{prop}
	\label{prop:causal}
	Let \(K:(H_{\mu}, \bH_{\mu})\to (H,\bH)\) be a bounded linear operator.
	Then \(K\) is causal if and only if  \(K\) maps \(H_{\mu,t}^{\bot}\) into \(H_{t}^{\bot}\) for any \(t\in[0,T]\).
\end{prop}
\begin{proof}
	Let us first assume that \(K\) is causal.
	For \(h\in H_{t}\), we write \(K^{*}(h)=f+g\), where \(f\in H_{\mu,t}^{\bot}\) and \(g\in H_{\mu,t}\).
	We notice that from the causality of \(K\) \[K^{-1}(\{x\in H:\la h,x\ra\leq 0\})=\{x\in H_{\mu}:\la h, K(x)\ra\leq 0\}=\{x\in H_{\mu}: \la f+g, x\ra_{\mu} \leq 0\}\in\cH_{\mu,t}.\]
	Since \(\cH_{\mu,t}=\sigma(\{h\in H_{\mu}:\supp (h)\subseteq [0,t]\})\),  we derive \(H_{\mu,t}^{\bot}\subseteq  U \) for any \(U\in \cH_{\mu,t}\).
	In particular, \(H_{\mu, t}^{\bot} \subseteq \{x\in H_{\mu}: \la f+g, x\ra_{\mu} \leq 0\} \) and hence \(f=0\).
	Therefore, \(P_{H_{\mu,t}^{\bot}}K^{*}P_{H_{t}}=0\).
	Taking the adjoint on both sides, we deduce \(P_{H_{t}}KP_{H_{\mu,t}^{\bot}}=0\), i.e., \(K\) maps \(H_{\mu,t}^{\bot}\) into \(H_{t}^{\bot}\).

	On the other hand, if \(H_{\mu,t}^{\bot}\) is mapped into \(H_{t}\) under \(K\) for any \(t\in[0,T]\), then \(H_{t}^{\bot}\) is mapped into \(H_{\mu,t}\) under \(K^{*}\).
	For any \(h\in H_{t}^{\bot}\) and \(r\in \bbR\), we have
	\begin{equation*}
		K^{-1}(\{x\in H:\la h,x\ra\leq r\})=\{x\in H_{\mu}:\la h, K(x)\ra_{\mu}\leq r\}=\{x\in H_{\mu}: \la K^{*}(h), x\ra_{\mu} \leq r\}.
	\end{equation*}
	The causality follows directly from the fact that \(K^{*}(h)\in H_{\mu,t}\).
\end{proof}
\begin{rmk}
	When \(\mu=\lambda\), the set of operators \(K: H_{\mu}\to H_{\mu}\) which leaves \(H_{\mu,t}^{\bot}\) invariant forms a non-selfadjoint algebra.
	This algebra is called the nest algebra first introduced in \citet{ringrose65Algebras}, and we denote it as \(\frN(H_{\mu})\).
	The diagonal algebra \(\frD(H_{\mu})\) is a subalgebra of \(\frN(H_{\mu})\) consisting of operators \(K\) such that both \(K\) and \(K^{*}\) are in \(\frN(H_{\mu})\). See \cite{davidson88Nest} for a detailed reference.
\end{rmk}

\subsection{Existence}
We investigate the existence of causal factorization.
Similar to Cholesky decomposition, it does exist for any \(\Sigma\in \frC(H)\).
The proof is based on a factorization result \citep[Theorem 13]{anoussis98Factorisation} in nest algebra.
\begin{prop}
	\label{prop-fact}
	Let \(\mu\) be a positive continuous measure, and \(\cR_{\mu}=\{\range(A): A\in \frN(H_{\mu})\}\). For any \(A\in \frB(H_{\mu})\), there exists \(B\in\frN(H_{\mu})\) such that \(AA^{*}= BB^{*}\) if and only if \(\range(A)\in \cR_{\mu}\).
\end{prop}

\begin{thm}
	\label{thm-exist}
	Let \(\Sigma\in \frC(H)\). Then there exists a causal factorization \((K,\mu)\) of \(\Sigma\).
\end{thm}

\begin{proof}
	Let \(X\) be a centered, mean-square continuous, and purely nondeterministic Gaussian process associated to \(\Sigma\).
	By Theorem \ref{thm-canonical}, we have a canonical representation of \(X\) given by
	\begin{equation}
		\label{eqn-rep}
		X(t)=\sum_{n=1}^{N}\int_{0}^{t} k^{n}(t,s) \md M^{n}(s).
	\end{equation}
	Recall we write \(\mu^{n}(\md t)=[M^{n}](\md t)\), and \(\mu^{1}\gg \mu^{2}\gg\cdots\).
	Let \(\mu=\lambda+\mu^{1}\).
	Since \(\Sigma\in \frC(H)\), we notice \(\Sigma\) uniquely determines a continuous kernel given by \(R(t,s)=\E[X(t)X(s)]\).
	Hence, it uniquely induces an operator \(\Sigma_{\mu}:H_{\mu}\to H_{\mu}\) given by
	\begin{equation*}
		\Sigma_{\mu}f(t)=\int_{0}^{T} f(s) R(t,s)\mu(\md s).
	\end{equation*}
	The representation \eqref{eqn-rep} yields a representation of \(\Sigma_{\mu}\) as
	\(
	\Sigma_{\mu}=\sum_{n=1}^{N} K^{n}(K^{n})^{*},
	\)
	where \(K^{n}:H_{\mu}\to H_{\mu}\) is given by
	\begin{equation*}
		K^{n}f(t)=\int_{0}^{t}k^{n}(t,s)\sqrt{\frac{\md \mu^{n}}{\md \mu}}(s)f(s)\mu(\md s).
	\end{equation*}
	In particular, \(K^{n}\in \frN(H_{\mu})\).
	If the multiplicity \(N\) was finite, then we could apply  \citet[Proposition 27]{anoussis98Factorisation} which states the sum of two factorizable operators can still be factored in the nest algebra.
	This would give us a \(K_{\mu}\in\frN(H_{\mu})\) such that \(\Sigma_{\mu}=K_{\mu}(K_{\mu})^{*}\).
	We could construct \(K:H_{\mu}\to H\) as \[Kf(t)=\int_{0}^{t}k_{\mu}(t,s)\sqrt{\frac{\md \lambda}{\md \mu}}(s)f(s)\mu(\md s),\]
	where \(k_{\mu}\) is the kernel of \(K_{\mu}\).
	Then it is direct to verify \((K,\mu)\) would be a causal factorization of \(\Sigma\).

	Now, we proceed with the case \(N=\infty\).
	The spirit of the proof aligns with \citet[Proposition 27]{anoussis98Factorisation}, but we extend it to a countable sum of operators.
	By Proposition~\ref{prop-fact}, it suffices to show \(\range({\Sigma_{\mu}^{1/2}})\in \cR_{\mu}\).
	We define \(T: \bigoplus_{n=1}^{\infty} H_{\mu}\to \bigoplus_{n=1}^{\infty} H_{\mu}\) as
	\begin{equation*}
		T(f_{1},f_{2},\dots)=\lb(\sum_{n=1}^{\infty}K^{n}f_{n},0,\dots\rb).
	\end{equation*}
	Notice that \(T\) is a bounded linear operator and \(\range(T)=\range((TT^{*})^{1/2})\) by \citet[Theorem 1]{douglas66majorization}.
	This yields \(\range(\Sigma_{\mu}^{1/2})=\sum_{n=1}^{\infty}\range(K^{n})\).
	We construct a sequence of partial isometries \(\{U_{n}\}_{n=1}^{\infty}\) in \(\frN(H_{\mu})\) with full range and mutually orthogonal initial spaces.
	Let \(e_{n}\in H_{\mu}\) with \(\supp(e_{n})\in[\frac{1}{n+1},\frac{1}{n}]\).
	We take \(\{I_{n,m}\}_{n,m=1}^{\infty}\) where  \(I_{n,m}\) are infinite and mutually disjoint subsets of \(\bbZ_{+}\).
	We define closed subspaces of \(H_{\mu}\) by \(E_{n,m}:=\overline{\mathrm{span}\{e_{k}: k\in I_{n,m}, k> m\}}\) and \(F_{m}:= \{f\in H_{\mu}: \supp(f)\subseteq [\frac{1}{m+1},\frac{1}{m}] \}\).
	In particular, \(E_{n,m}\) are mutually orthogonal.
	Since \(E_{n,m}\) is infinite dimensional, we can find a partial isometry \(P_{n,m}\in\frN(H_{\mu})\) with initial space \(E_{n,m}\) and range \(F_{m}\).
	By taking \(U_{n}= \sum_{m=1}^{\infty} P_{n,m}\), we have \(U_{n}\in \frN(H_{\mu})\) with full range and mutually orthogonal initial spaces.
	Therefore, \(\sum_{n=1}^{\infty}\range(K^{n})=\range\lb(\sum_{n=1}^{\infty}K^{n}U_{n}\rb)\in \cR_{\mu}\).
	We conclude the proof by applying Proposition~\ref{prop-fact} and notice \(\range(\Sigma_{\mu}^{1/2})\in \cR_{\mu}\).
\end{proof}

We say a process \(X\) is Gaussian Volterra if there exists a Volterra representation \(X(t)=\int_{0}^{t}k(t,s)\md M(s)\) with \(M\) a continuous Gaussian martingale with independent increments.
The above result gives a characterization of mean-square continuous Gaussian Volterra processes.
\begin{cor}
	\label{cor-volterra}
	Let \(X\) be a centered, mean-square continuous Gaussian process.
	The following statements are equivalent:
	\begin{enumerate}[label=(\roman*)]
		\item \(\cF_{0+}^{X}=\bigcap_{t>0}\cF_{t}^{X}\) is trivial.
		\item \(X\) is purely nondeterministic.
		\item There exists a Gaussian Volterra process \(\tilde{X}\) such that \(X\) and \(\tilde{X}\) share the same law.
	\end{enumerate}
\end{cor}
\begin{proof}
	(i) \(\Rightarrow\) (ii).
	Notice that \(\overline{\mathrm{span}\{X(s):s\leq t\}}\subseteq \cF_{t}^{X}\).
	Therefore, \(\cF_{0+}^{X}\) is trivial implies that \, \(\bigcap_{t>0}\overline{\mathrm{span}\{X(s):s\leq t\}}=0\), and hence \(X\) is purely nondeterministic.

	(ii) \(\Rightarrow\) (iii).
	If \(X\) is purely nondeterministic, by Theorem \ref{thm-exist} there exists a causal factorization \((K,\mu)\) of \(\Sigma\), the covariance operator of \(X\).
	We can take a Gaussian Volterra process \(\tilde{X}(t)=\int_{0}^{t}k(t,s)\md M(s)\), where \(k\) is the kernel of \(K\), and \(M\) is a Gaussian martingale with independent increments and \(\mu(\md t)=[M](\md t)\).
	It is clear that \(\tilde{X}\) has the same covariance operator as \(X\), and hence they share the same law.

	(iii) \(\Rightarrow\) (i). \(\tilde{X}(t)=\int_{0}^{t}k(t,s)\md M(s)\) is a Gaussian Volterra process and shares the same law as \(X\).
	In particular, \(M\) is continuous Gaussian martingale with independent increments, and it is a deterministic continuous time change of the standard Brownian motion.
	Therefore, \(\cF_{0+}^{\tilde{X}}\subseteq\cF_{0+}^{M}\) is trivial, and so as \(\cF_{0+}^{X}\).


\end{proof}

On the other hand, a canonical causal factorization does not always exist.
In particular, the following result links the canonical causal factorization to Gaussian processes with unit multiplicity.
\begin{thm}
	Let \(\Sigma\in \frC(H)\) and be associated to a Gaussian process \(X\).
	The following statements are equivalent:
	\begin{enumerate}[label=(\roman*)]
		\item \(X\) is of unit multiplicity and has a canonical representation \(X(t)=\int_{0}^{t}k(t,s)\md M(s)\).
		\item \(\Sigma\) has a canonical causal factorization \((K,\mu)\).
		\item \(\Sigma\) has a causal factorization \((K,\mu)\) such that  the span of \(\{k(r,\cdot): r\in[0,t]\}\) is dense in \(H_{\mu,t}\) for any \(t\in[0,T]\), where \(k\) is the kernel of \(K\).
	\end{enumerate}
\end{thm}
\begin{proof}
	\((i)\Leftrightarrow (ii)\).
	We notice by Theorem \ref{thm-can}, \(X(t)=\int_{0}^{t}k(t,s)\md M(s) \) is a canonical representation, if and only if
	for any \(T'\in[0,T]\) and \(f\in H_{\mu}\),
	\begin{equation*}
		g(t)=\int_{0}^{t} k(t,s)f(s)\mu(\md s)=0 \text{ for all } t\in[0,T']
	\end{equation*}
	implies \(f=0\) on \([0,T']\).
	This is equivalent to the injectivity of \(K_{T'}=P_{H_{T'}}K|_{H_{\mu_{T'}}}\) for any \(T'\in[0,T ]\) where \(K\) is given by \(K f(t)=\int_{0}^{t}k(t,s)f(s) \mu(\md s)\).
	Since \(K\in \frN(H_{\mu})\), this is further equivalent to \((K,\mu)\) is a canonical causal factorization of \(\Sigma\).

	\((ii)\Leftrightarrow (iii)\). Notice \(K_{t}=P_{H_{t}}K|_{H_{\mu,t}}\) is injective if and only if the range of \(K_{t}^{*}\) is dense in \(H_{\mu,t}\).
	Since \(\Sigma\in \frC (H)\), \(R(t,s)=\E[X(t)X(s)]= \int_{0}^{t\wedge s}k(t,r)k(s,r) \mu(\md r)\) is continuous.
	This implies \(r\mapsto k(r,\cdot)\in H_{\mu}\) is continuous.
	Therefore,  \(\mathrm{range}(K_{t}^{*})= \{f(s)=\int_{0}^{t}k(r,s)g(r)\md r: g\in H_{t}\}\) is dense if and only if the span of \(\{k(r,\cdot): r\in[0,t]\}\) is dense in \(H_{\mu,t}\)
\end{proof}

\subsection{Uniqueness}

In the finite dimensional case, Cholesky decomposition is unique up to a diagonal matrix.
This is saying for nondegenerate, lower-triangular matrices \(K_1,K_2\) satisfying \(K_{1}K_{1}^{*}=K_{2}K_{2}^{*}\), there exists a diagonal matrix \(D\) such that \(K_{1}=K_{2}D\).
However, it is not the case for the causal factorization.
\begin{prop}
	\label{prop-nonuniq}
	Let \(\Sigma\in \frC(H)\) and \((K,\mu)\) be a causal factorization of \(\Sigma\).
	For any partial isometry \(U\in \frN(H_{\mu})\)  with range dense in \(H_{\mu}\), \((KU, \mu)\) is again a causal factorization of \(\Sigma\).
	Moreover, \(U\) is not necessarily in the diagonal algebra \(\frD(H_{\mu})\), i.e., \(U^{*}\) is not necessarily in \(\frN(H_{\mu})\).
\end{prop}
\begin{proof}
	Since \(U\) is a partial isometry with a dense range, we have \(UU^{*}=\Id\) on \(H_{\mu}\).
	Therefore, \((KU,\mu)\) is a causal factorization of \(\Sigma\).
	For example, we can take \(U\) as in the proof of Theorem \ref{thm-exist}.
	And in particular, \(U\) is not diagonal.
\end{proof}

Such non-uniqueness generates non-canonical representations of the same Gaussian process.
In the following example, we include Levy's non-canonical representation of Brownian motion \citep{levy57Fonctions}.
\begin{exam}
	\label{exam-levy}
	We define a partial isometry \(U\) on \(H\) given by \(U^{*}\1_{[0,t]}(s)= [3-12(s/t)+10(s/t)^{2}]\1_{[0,t]}(s) \).
	It is direct to verify \(X(t)=\int_{0}^{t}\{3-12 (s/t)+10(s/t)^{2}\}\md B(s)\) is a Brownian motion.
	However, this representation is not canonical.
	Notice \(X(t)\) is independent of \(\int_{0}^{T}s\md B(s)\), which implies \(\cF_{T}^{X}\subsetneq \cF_{T}^{B}\).
\end{exam}

If we restrict ourselves to the canonical causal factorization, we retrieve a uniqueness result analogous to the one for Cholesky decomposition.

\begin{prop}
	\label{prop-uniq}
	Let \(\Sigma\in \frC(H)\).
	Assume \((K_{1},\mu)\) and \((K_{2},\mu)\) are two canonical causal factorizations of \(\Sigma\).
	Then,  there exists a diagonal operator \(D\in\frD(H_{\mu})\) such that \(K_{1}=K_{2}D\).
	Moreover, \(D\) is a multiplication operator given by \(D f (t)= (\1_{S}(t)-\1_{[0,T]\setminus S}(t))f(t) \) for a measurable set \(S\subseteq [0,T]\).
\end{prop}
\begin{proof}
	Since \((K_{2},\mu)\) is canonical, we have \(K_{2}\) is injective and hence \(\overline{\range(K_{2}^{*})}=H_{\mu}\).
	Since \(K_{1}K_{1}^{*}=K_{2}K_{2}^{*}\), we  deduce \(K_{1}^{*}\) and \(K_{2}\) share the same null space.
	We can define an operator \(\tilde{D}\) from \(\range{K_{2}^{*}}\) to \(\range(K_{1}^{*})\) such that \(\tilde{D}(K_{2}^{*}f)=K_{1}^{*}f\).
	Moreover, \(\tilde{D}\) can be uniquely extended to an operator on \(H_{\mu}=\overline{\range(K_{2}^{*})}\).
	Therefore, by taking \(D=\tilde{D}^{*}\), we derive \(K_{1}=K_{2}D\).
	Noticing \( K_{2}DK_{1}^{*}=K_{1}K_{1}^{*}=K_{2}K_{2}^{*}\) and \(K_{2}\) is injective, we deduce \(DK_{1}^{*}=K_{2}^{*}\) and \(D^{*}=K_{1}^{-1}K_{2}\).
	This yields that \(D\) is an orthogonal operator on \(H_{\mu}\).

	Now, we consider two canonical representations induced by \(K_{1}\) and \(K_{2}\)
	\begin{equation*}
		X(t)=\int_{0}^{T}k_{1}(t,s)\md M_{1}(s)=\int_{0}^{T}k_{2}(t,s)\md M_{2}(s),
	\end{equation*}
	where \(k_{i}\) is the kernel of \(K_{i}\).
	Since  \(M_{1}\) and \(M_{2}\) generate the same filtration as \(X\), \(M_{1}\) is a \(\bF^{M_{2}}\)--martingale.
	Moreover, \(M_{2}\) is a continuous Ocone martingale with deterministic quadratic variation.
	Therefore, by martingale representation theorem \citet[Proposition]{vostrikova2007some}, we have \(M_{1}(t)=\int_{0}^{t}\rho(s)\md M_{2}(s)\) for some predictable process \(\rho(s)\) taking value in \(\{-1,1\}\).
	Together with the fact that \(K_{2}=K_{1}D^{*}\), we deduce \(k_{2}(t,\cdot)=D k_{1}(t,\cdot)=\rho(\omega,\cdot ) k_{1}(t,\cdot)\in H_{\mu}\), \(P(\md \omega)\)-a.s \(\lambda(\md t)\)-a.e.
	This implies \(\rho(\omega,\cdot)\in H_{\mu}\) is deterministic and has the form of \(\rho(s)= \1_{S}(s)-\1_{[0,T]-S}(s)\).
	Otherwise,   \(k_{1}(t,\cdot)=0\) on a positive measure set  which contradicts the injectivity of \(K_{1}\).
	Moreover, we notice the span of \(\{ k_{1}(t,\cdot): t\in[0,T]\}\) is dense in \(H_{\mu}\) as \(K_{1}\) is injective, and we conclude \(D f (t)= f(t) (\1_{S}(t)-\1_{[0,T]-S}(t))\).
\end{proof}

\begin{rmk}
	Following the same lines of arguments, we can show that all orthogonal operators \(O\)  in the nest algebra \(\frN(H_\mu)\) with \(P_{H_{\mu,t}}O|_{H_{\mu,t}}\) surjective for all \(t\in[0,T]\) are diagonal.
	This is of sharp contrast to the result of \citet{davidson98RussoDye} which shows the abundance of the unitary operators in a nest algebra on a complex Hilbert space.
	Indeed under their setting, any contraction in \(\frN(H_{\mu})\) can be represented as a finite convex combination of unitary operators, and hence there are non-diagonal unitary operators in \(\frN(H_{\mu})\).
\end{rmk}

\section{Gaussian processes}
\label{sec-Gaussian}
Before we present our main theorem, we show that we can always decompose a Gaussian process into a deterministic part and a purely nondeterministic part.
These two parts are `orthogonal', and we can calculate the adapted Wasserstein distance separately.
\begin{lem}
	For any mean-square continuous Gaussian process \(X\), there exists a decomposition \(X=Y+Z\) where \(Y\) is purely nondeterministic and \(Z\) is deterministic.
	Moreover, \(Y\) and \(Z\) are independent mean-square continuous Gaussian processes.
	The adapted Wasserstein distance between \(X_{1}\) and \(X_{2}\) can be decomposed as
	\begin{equation*}
		\AW_{2}(X_{1},X_{2})^{2}= \AW_{2}(Y_{1},Y_{2})^{2}+\W_{2}(Z_{1},Z_{2})^{2}.
	\end{equation*}
\end{lem}

\begin{rmk}
	The Wasserstein distance between two Gaussian processes \(Z_{1}\) and \(Z_{2}\) is well studied (see \citet{dowson82Frechet,gelbrich90Formula}), and can be calculated explicitly given the covariance operators of \(Z_{1}\) and \(Z_{2}\).
\end{rmk}

\begin{proof}
	The first statement is a generalization of Wold decomposition to general second order stochastic processes, see \citet{cramer71Structural}.
	The deterministic process \(Z\) is given by \(Z(t)=P_{0+}X(t)\) where \(P_{t}\) is the orthogonal projection from \(L^{2}(\Omega, P)\) to the closed subspace
	\(\overline{\mathrm{span}\{X_{s}:0\leq s\leq t\}}\) and \(P_{0+}=\lim_{t\to 0+}P_{t}\).
	It is direct to verify \((X,Z)\) is jointly Gaussian and so is \((Y,Z)\).
	Therefore, the independence of \(Y\) and \(Z\) follows from the orthogonality of the projection.
	Since mean-square continuity can be preserved by the orthogonal projection, we have \(Y\) and \(Z\) are mean-square continuous.

	We proceed to show the decomposition of the adapted Wasserstein distance between \(X_{1}\) and \(X_{2}\).
	Noticing under any bicausal coupling \(\pi\in\Pi_{\bc}(X_{1},X_{2})\), \(\cF^{X_{1}}_{t}\) is conditionally independent of \(\cF^{X_{2}}_{T}\) given \(\cF^{X_{2}}_{t}\).
	This implies \(\cF^{Z_{1}}_{T}=\cF^{Z_{1}}_{t}\) is conditionally independent of \(\cF^{Y_{2}}_{T}\) given \(\cF^{X_{2}}_{t}\).
	Hence, we deduce
	\begin{align*}
		\E_{\pi}[\la Z_{1},Y_{2}\ra]= \E_{\pi}[\E_{\pi}[\la Z_{1},Y_{2}\ra|\cF^{X_{2}}_{t}]]=\E_{\pi}[\la \E_{\pi}[ Z_{1}|\cF^{X_{2}}_{t}],\E_{\pi}[Y_{2}|\cF^{X_{2}}_{t}]\ra].
	\end{align*}
	Notice that \(\E_{\pi}[Y_{2}(\cdot)|\cF_{t}^{X_{2}}]= \E_{P}[X_{2}(\cdot)-Z_{2}(\cdot)|\cF_{t}^{X_{2}}]= (P_{t}-P_{0+})X_{2}(\cdot)\).
	By Lebesgue dominated convergence theorem, we derive \(\E_{\pi}[\la Z_{1},Y_{2}\ra]=0\) by taking \(t\) to \(0\).
	Therefore, we have \(\AW_{2}(X_{1},X_{2})^{2}\geq \AW_{2}(Y_{1},Y_{2})^{2}+\AW_{2}(Z_{1},Z_{2})^{2}\).
	Finally, noticing for deterministic process \(Z_{i}\), it holds \(\cF^{Z_{i}}_{0+}=\cF^{Z_{i}}_{T}\), and hence \(\AW_{2}(Z_{1},Z_{2})=\W_{2}(Z_{1},Z_{2})\).

	For the reverse direction, we consider the optimal bicausal coupling \(\pi_{Y}(\pi_{Z})\) attaining the adapted Wasserstein distance between \(Y_{1}\) and \(Y_{2}\) (\(Z_{1}\) and \(Z_{2}\)).
	Then we construct a bicausal coupling from the independent product \(\pi_{Y}\otimes\pi_{Z}\).
	Let \(\hat{\pi}=  (Y_{1}+Z_{1},Y_{2}+Z_{2})_{\#}(\pi_{Y}\otimes \pi_{Z})\).
	By Remark~\ref{rmk-ad}, there exists \(\pi\in \Pi_{\bc}(X_{1},X_{2})\) such that \((X_{1},X_{2})_{\#}\pi= \hat{\pi}\).
	Hence, this yields  \(\AW(X_{1},X_{2})^{2}\leq \E_{\pi}[\|X_{1}-X_{2}\|^{2}]=\AW_{2}(Y_{1},Y_{2})^{2}+\W_{2}(Z_{1},Z_{2})^{2}\).

\end{proof}

\subsection{Unit multiplicity}
We present an explicit adapted Wasserstein distance formula for Gaussian processes of unit multiplicity.
\begin{thm}
	\label{thm-gp}
	Let \(X_{i}\) be a centered,  mean-square continuous, and purely nondeterministic Gaussian process of unit multiplicity, with canonical representation \(X_{i}(t)=\int_{0}^{t}k_{i}(t,s)\md M_{i}(s)\) for \(i=1,2\).
	Then, the adapted Wasserstein distance between \(X_{1}\) and \(X_{2}\) is given by
	\begin{equation}
		\label{eqn-gp1}
		\AW_{2}(X_{1},X_{2})^{2}=\int_{0}^{T}\| k_{1}(\cdot,s)\|^{2}\mu_{1}(\md s)+\int_{0}^{T}\| k_{2}(\cdot,s)\|^{2}\mu_{2}(\md s)- 2\int_{0}^{T}\lb |\la k_{1}(\cdot,s), k_{2}(\cdot, s)\ra\rb| \sqrt{\mu_{1}\mu_{2}}(\md s),
	\end{equation}
	where \(\mu_{i}(\md s)=[M_{i}](\md s)\).

	Equivalently, let \(\Sigma_{i}\) be the covariance operator of \(X_{i}\), \((K_{i},\mu_{i})\) be a canonical causal factorization of \(\Sigma_{i}\).
	We have the adapted Wasserstein distance
	\begin{equation}
		\label{eqn-gp2}
		\AW_{2}(X_{1},X_{2})^{2}=\tr(\Sigma_{1}+\Sigma_{2})-2\int_{0}^{T} \|\md H_{\mu_{1}} K_{1}^{*}K_{2}\md H_{\mu_{2}}\|_{\HS},
	\end{equation}
	where
	\begin{equation*}
		\int_{0}^{T} \|\md H_{\mu_{1}} K_{1}^{*}K_{2}\md H_{\mu_{2}}\|_{\HS}:=\lim_{\|\bP\|\to 0}\sum_{(s,t)\in \bP} \|(P_{\mu_{1},t}-P_{\mu_{1},s})K_{1}^{*}K_{2}(P_{\mu_{2},t}-P_{\mu_{2},s})\|_{\HS},
	\end{equation*}
	and \(P_{\mu_{i},t}\) denotes the projection of \(H_{\mu_{i}}\) to the subspace \(H_{\mu_{i},t}=\{f\in H_{\mu_{i}}:\supp(f)\subseteq[0,t]\}\).
	Here, the limit is taken over all partitions \(\bP\) of \([0,T]\) with mesh size \(\|\bP\|\) converging to 0.
\end{thm}

\begin{rmk}
	The distance does not depend on the choice of the canonical representation.
	Indeed, if \(k_{1}\) and \(\tilde{k}_{1}\) are kernels of two canonical representations of \(X_{1}\), by Proposition \ref{prop-uniq} we have \(\tilde{k}_{1}(\cdot,s)=k_{1}(\cdot,s)(\1_{S}(s)- \1_{[0,T]\setminus S}(s))\).
	Hence, plugging \(\tilde{k}_{1}\) into \eqref{eqn-gp1} does not change its value.
\end{rmk}

\begin{rmk}
	One shall not expect to relax the condition of the canonical representation.
	We consider the noncanonical representation of Brownian motion given in Example \ref{exam-levy}.
	Naively plugging in the formula, we would obtain a positive quantity for the adapted Wasserstein distance between two standard Brownian motions.
\end{rmk}

\begin{rmk}
	Although we focus on mean-square continuous Gaussian processes, the proof can be easily adapted to  the discrete-time case.
	Moreover, \eqref{eqn-gp2} is consistent with the discrete-time result given in \citet{gunasingam25Adapteda}.
	In discrete-time case, the triangular integral \(\int_{0}^{T} \|\md H_{\mu_{1}} K_{1}^{*}K_{2}\md H_{\mu_{2}}\|_{\HS}\) can be interpreted as the sum of the diagonal elements of \(K_{1}^{*}K_{2}\).
	Here, the notation of triangular integral is adapted from the literature of nest algebra, e.g., \citet{davidson88Nest}.
\end{rmk}

\begin{proof}[Proof of Theorem \ref{thm-gp}]
	Since \(\bF^{X_{i}}=\bF^{M_{i}}\), by definition we obtain \(\Pi_{\bc}(X_{1},X_{2})=\Pi_{\bc}(M_{1},M_{2})\).
	We apply the transfer principle and derive that
	\begin{align*}
		\sup_{\pi\in \Pi_{\bc}(X_{1},X_{2})} \E_{\pi}[\la X_{1} ,X_{2}\ra] & = \sup_{\pi\in \Pi_{\bc}(M_{1},M_{2})}\E_{\pi}[\la X_{1} ,X_{2}\ra]                                                             \\
		                                                                   & =\sup_{\pi\in \Pi_{\bc}(M_{1},M_{2})}  \E_{\pi}\lb[\int_{0}^{T} \int_{0}^{T}k_{1}(t,s)k_{2}(t,s) [M_{1},M_{2}](\md s)\md t\rb].
	\end{align*}
	The second equality follows from the fact that \(M_{1}\) and \(M_{2}\) remain martingales with respect to the product filtration under any bicausal coupling, see \citet[Remark 2.3]{acciaio20Causal}.
	By Fubini theorem and Kunita--Watanabe inequality, we derive
	\begin{align*}
		\sup_{\pi\in \Pi_{\bc}(X_{1},X_{2})} \E_{\pi}[\la X_{1} ,X_{2}\ra] & =\sup_{\pi\in \Pi_{\bc}(M_{1},M_{2})}  \E_{\pi}\lb[\int_{0}^{T} \int_{0}^{T}k_{1}(t,s)k_{2}(t,s) \md t [M_{1},M_{2}](\md s)\rb]             \\
		                                                                   & =\sup_{\rho(\cdot)\in[-1,1]}\E_{\pi}\lb[\int_{0}^{T}\lb(\int_{0}^{T} k_{1}(t,s)k_{2}(t,s)\md t \rb)\rho(s) \sqrt{\mu_{1}\mu_{2}}(\md s)\rb] \\
		                                                                   & =\int_{0}^{T}\lb|\la k_{1}(\cdot ,s),k_{2}(\cdot ,s)\ra\rb| \sqrt{\mu_{1}\mu_{2}}(\md s).
	\end{align*}
	The second equality follows from the fact that \( \sqrt{\mu_{1}\mu_{2}} \gg[M_{1},M_{2}]\) and the Radon--Nikodym density \(\rho\) takes values in \([-1,1]\).
	Moreover, the optimal bicausal coupling is induced by a Gaussian coupling \[M_{1}(t)=\int_{0}^{t}\sqrt{\frac{\md \mu_{1}}{\md (\mu_{1}+\mu_{2})}}(s)\md \tilde{M}(s) \quad\text{and } M_{2}(t)=\int_{0}^{t}\sqrt{\frac{\md \mu_{2}}{\md (\mu_{1}+\mu_{2})}}(s)\rho(s)\md \tilde{M}(s),\] where  \(\rho\) attains the supremum in the above estimate and \(\tilde{M}\) is a Gaussian martingale with independent increments, \([\tilde{M}](\md s)=(\mu_{1}+\mu_{2})(\md s)\).

	For \eqref{eqn-gp2}, we fix \(s,t \in[0,T]\).
	We notice  for any \(f\in H_{\mu_{2}}\)
	\begin{align*}
		(P_{\mu_{1},t}-P_{\mu_{1},s})K_{1}^{*}K_{2}(P_{\mu_{2},t}-P_{\mu_{2},s})f (r_{1}) =\int_{s}^{t}\la k_{1}(\cdot,r_{1 }),k_{2}(\cdot,r_{2})\ra f(r_{2})\mu_{2}(\md r_{2}).
	\end{align*}
	This gives
	\begin{equation*}
		\|(P_{\mu_{1},t}-P_{\mu_{1},s})K_{1}^{*}K_{2}(P_{\mu_{2},t}-P_{\mu_{2},s})\|_{\HS} = \lb[\int_{s}^{t}\int_{s}^{t}|\la k_{1}(\cdot, r_{1}),k_{2}(\cdot,r_{2})\ra|^{2}\mu_{1}(\md r_{1})\mu_{2}(\md r_{2} )\rb]^{1/2}.
	\end{equation*}
	Since \(X_{i}\) is mean-square continuous, we have \((r_{1},r_{2})\mapsto \la k_{1}(\cdot, r_{1}),k_{2}(\cdot,r_{2})\ra\) is uniformly continuous on \([0,T]\times [0,T]\).
	This allows us to conclude
	\begin{equation*}
		\int_{0}^{T} \|\md H_{\mu_{1}} K_{1}^{*}K_{2}\md H_{\mu_{2}}\|_{\HS}= \int_{0}^{T}\lb|\la k_{1}(\cdot ,s),k_{2}(\cdot ,s)\ra\rb| \sqrt{\mu_{1}\mu_{2}}(\md s).
	\end{equation*}
\end{proof}

We give several examples.

\begin{exam}
	We consider the adapted Wasserstein distance between a standard Brownian motion \(B\) and a Cantor Gaussian martingale \(C\).
	The covariance operator of the Cantor Gaussian martingale \(C\) is given by \(E[C(t)C(s)]=F(t\wedge s)\), where \(F\) is the Cantor function, also known as the Devil's staircase.
	In particular, \(F(\md t)\) is mutually singular to the Lebesgue measure.
	This implies that under any bicausal coupling \(B(t)\) and \(C(t)\) are uncorrelated which gives
	\begin{equation*}
		\AW_{2}(B,C)^{2}=\int_{0}^{T}(t+ F(t))\md t.
	\end{equation*}
	In fact, every bicausal coupling attains the adapted Wasserstein distance.
	On the other hand, one can easily construct a non-bicausal coupling by the time change of Brownian motion under which \(B\) and \(C\) are not independent anymore and have a transport cost strictly less than \(\int_{0}^{T}(t+F(t))\md t\).

\end{exam}

\begin{exam}
	We consider the adapted Wasserstein distance between two fractional Brownian motions.
	For a fractional Brownian motion \(B_{H}\) with Hurst parameter \(H\), it has a stochastic representation given by
	\begin{equation*}
		B_{H}(t)=\int_{0}^{t}k_{H}(t,s)\md B(s),
	\end{equation*}
	where  \(k_{H}\) is the Molchan--Golosov kernel, see \citet{molchan69Gaussian,decreusefond99Stochastic} for example.
	In particular, this gives a canonical representation of \(B_{H}\), see \citet[Theorem 5.1]{jost06Transformation}.
	Therefore, plugging this canonical representation into Theorem \ref{thm-gp}, we obtain Theorem \ref{thm-fb} and have
	\begin{equation*}
		\AW_{2}(B_{H_{1}},B_{H_{2}})^{2}=\int_{0}^{T}\int_{0}^{T}(k_{H_{1}}(t,s)-k_{H_{2}}(t,s))^{2}\md t\md s=\|K_{H_{1}}-K_{H_{2}}\|_{\HS}^{2}.
	\end{equation*}
	We remark that the synchronous coupling is the unique optimal bicausal coupling.

\end{exam}

\begin{exam}
	We consider the adapted Wasserstein distance between fractional Ornstein--Uhlenbeck processes given by
	\begin{equation*}
		X_{i}(0)=x_{i}- \lambda_{i}\int_{0}^{t}X_{i}(s)\md s + B_{H_{i}}(t),
	\end{equation*}
	whose solution is given by
	\begin{equation*}
		X_{i}(t)=\me^{-\lambda_{i} t} x_{i} + \int_{0}^{t}\me^{\lambda_{i}(s-t)}\md B_{H_{i}}(s).
	\end{equation*}
	Let \(\tilde{X}_{i}(t)=X_{i}(t)-\me^{-\lambda_{i} t} x_{i}\).
	Then, \(\tilde{X}_{i}\) is a centered Gaussian process, and \[\AW_{2}(X_{1},X_{2})^{2}=\AW_{2}(\tilde{X}_{1},\tilde{X}_{2})^{2}+\int_{0}^{T}|\me^{-\lambda_{1} t} x_{1}-\me^{-\lambda_{2} t} x_{2}|^{2}\md t.\]
	By \citet[Proposition A.1]{cheridito03Fractional}, we can show \(\tilde{X}_{i}\) is of unit multiplicity and with a canonical representation given by
	\begin{align*}
		\tilde{X}_{i}(t)= \int_{0}^{t}\me^{\lambda_{i}(s-t)}\md B_{H_{i}}(s) & = \int_{0}^{t}\lb( k_{H_{i}}(t,s)+\int_{s}^{t}\me^{\lambda_{i}(t-r)}k_{H_{i}}(r,s)\md r\rb ) \md B(s) \\
		                                                                     & :=\int_{0}^{t}k_{OU_{i}}(t,s)\md B(s).
	\end{align*}
	By Theorem \ref{thm-gp}, we derive \(\AW_{2}(\tilde{X}_{1},\tilde{X}_{2})^{2}=\int_{0}^{T}\int_{0}^{T}(k_{OU_{1}}(t,s)-k_{OU_{2}}(t,s))^{2}\md t\md s\) as \(k_{OU_{i}}\geq 0\).
\end{exam}

\subsection{Higher multiplicity}
We can also extend the result to the case of higher multiplicity.
\begin{thm}
	\label{thm-gp2}
	Let \(X_{1}\) and \(X_{2}\) be two centered,  mean-square continuous, and purely-nondeterministic Gaussian processes with canonical representations \[X_{1}(t)=\sum_{i=1}^{m}\int_{0}^{t}k_{1}^{i}(t,s)\md M_{1}^{i}(s) \quad \text{ and } \quad X_{2}(t)=\sum_{j=1}^{n}\int_{0}^{t}k_{2}^{j}(t,s)\md M_{2}^{j}(s).\]
	Then, the adapted Wasserstein distance between \(X_{1}\) and \(X_{2}\) is given by
	\begin{align*}
		\AW_{2}(X_{1},X_{2})^{2} & = \sum_{i=1}^{m}\int_{0}^{T}\| k_{1}^{i}(\cdot,s)\|^{2}\mu_{1}^{i}(\md s) +\sum_{j=1}^{n}\int_{0}^{T} \|k_{2}^{j}(\cdot,s)\|^{2}\mu_{2}^{j}(\md s) \\
		                         & \qquad -2 \int_{0}^{T}\lb\| \la \tilde{k}_{1}^{i}(\cdot, s), \tilde{k}_{2}^{j}(\cdot, s)\ra_{i,j}\rb\|_{\tr} \sqrt{\mu_{1}^{1}\mu_{2}^{1}}(\md s),
	\end{align*}
	where \(\tilde{k}_{1}^{i}(\cdot, s)=\sqrt{\dfrac{\md\mu_{1}^{i}}{\md\mu_{1}^{1}}}(s)k_{1}^{i}(\cdot,s)\) and \(\tilde{k}_{2}^{j}(\cdot, s)=\sqrt{\dfrac{\md\mu_{2}^{j}}{\md \mu_{2}^{1}}}(s)k_{2}^{j}(\cdot,s)\).

\end{thm}
\begin{rmk}
	We point out that even though the  Gaussian process \(X_{1}\) is one-dimensional, its natural filtration is `multi-dimensional'.
	Indeed, we can use \(X_{1}\) to reconstruct a multi-dimensional Gaussian martingale \(M_{1}=(M_{1}^{1},\dots, M_{1}^{m})\) with independent components, sharing the same natural filtration as \(X_{1}\).
	Hence, the adapted Wasserstein distance between higher multiplicity Gaussian processes is similar to the discrete-time multi-dimensional case \citep{acciaio24Entropica} where a trace norm is present.
	In the same fashion, one can derive the adapted Wasserstein distance between multi-dimensional Gaussian processes with arbitrary multiplicity.
	For brevity, we only present the one-dimensional case.
\end{rmk}
\begin{rmk}
	Gaussian processes with higher multiplicity do exist in theory, although they are mostly pathological and not common in practice.
	For example, the independent sum of a standard Brownian motion and a fractional Brownian motion with \(H>3/4\) is equivalent to a standard Brownian motion \citep{cheridito01Mixed}, and hence the mixture is still a Gaussian process of unit multiplicity \citet[Theorem 6.3]{hida93Gaussian}.
	In \citet[Chapter 4]{hida93Gaussian}, a Gaussian process with multiplicity 2 is constructed explicitly by taking \(X(t)=B_{1}(t)+F(t)B_{2}(t)\), where \(B_{1}\), \(B_{2}\) are independent standard Brownian motions, and \(F'\) is integrable but \(F\) is nowhere locally square integrable.
\end{rmk}
The following is an elementary algebraic lemma which we require for the proof of Theorem \ref{thm-gp2}.
\begin{lem}
	\label{lem-tr}
	Let \(A\in\bbR^{m\times m}\), \(B\in \bbR^{n\times n}\), and  \(C\in \bbR^{m\times n}\).
	Assume \(A\) and \(B\) are  semi-positive definite.
	Then, for any \(\Gamma\in \bbR^{m\times n}\) such that  \(\begin{psmallmatrix} A & \Gamma \\ \Gamma^{*} & B\end{psmallmatrix}\geq 0\) we have
	\begin{equation*}
		\tr(C\Gamma^{*})\leq \|A^{1/2} C B^{1/2}\|_{\tr}.
	\end{equation*}
	Moreover, the equality can be attained by \(\Gamma= A^{1/2}UVB^{1/2}\) where \(U\) and \(V\) are given by the singular value decomposition \(A^{1/2}CB^{1/2}=U\Sigma V\).
\end{lem}
\begin{proof}
	We first show the results for nondegenerate \(A\) and \(B\).
	We notice \(\begin{psmallmatrix} A & \Gamma \\ \Gamma^{*} & B\end{psmallmatrix}\geq 0\) is equivalent to \(I\geq (A^{-1/2}\Gamma B^{-1/2})^{*}(A^{-1/2}\Gamma B^{-1/2})\).
	Moreover, the singular value decomposition gives
	\begin{equation*}
		\tr(C\Gamma^{*})=\tr(A^{1/2}CB^{1/2} (A^{-1/2}\Gamma B^{-1/2})^{*} )\leq \tr(\Sigma)=\|A^{1/2}CB^{1/2}\|_{\tr}.
	\end{equation*}
	Now we consider the general case.
	Since \(\begin{psmallmatrix} A & \Gamma \\ \Gamma^{*} & B\end{psmallmatrix}\geq 0\) is equivalent to \(\begin{psmallmatrix} A_{\varepsilon} & \Gamma \\ \Gamma^{*} & B_{\varepsilon}\end{psmallmatrix}\geq 0\) for any \(\varepsilon>0\) where \(A_{\varepsilon}= A+\varepsilon I\) and \(B_{\varepsilon}=B+\varepsilon I\).
	We derive \(\tr(C\Gamma^{*})\leq \|A_{\varepsilon}^{1/2}CB_{\varepsilon}^{1/2}\|_{\tr}\) for any \(\varepsilon>0\).
	Therefore, we conclude the proof by taking the limit \(\varepsilon\to 0\) and noticing the equality can be attained by \(\Gamma=A^{1/2}UVB^{1/2}\).
\end{proof}

\begin{proof}[Proof of Theorem \ref{thm-gp2}]
	We put emphasis on the difference between the unit multiplicity case and the higher multiplicity case and only sketch the similar part.
	We write \(M_{1}=(M_{1}^{1},\dots, M^{m}_{1})\), \(M_{2}=(M_{2}^{1},\dots,M_{2}^{n})\).
	Similar to the unit multiplicity case we notice
	\begin{align*}
		\sup_{\pi\in \Pi_{\bc}(X_{1},X_{2})} \E_{\pi}[\la X_{1} ,X_{2}\ra] & = \sup_{\pi\in \Pi_{\bc}(M_{1},M_{2})}\E_{\pi}[\la X_{1} ,X_{2}\ra]                                                                                                                \\
		                                                                   & =\sup_{\pi\in \Pi_{\bc}(M_{1},M_{2})}  \sum_{i,j}\E_{\pi}\lb[\int_{0}^{T} \int_{0}^{T}k_{1}^{i}(t,s)k_{2}^{j}(t,s) [M_{1}^{i},M_{2}^{j}](\md s)\md t\rb]                           \\
		                                                                   & = \sup_{\pi\in \Pi_{\bc}(M_{1},M_{2})} \sum_{i,j} \E_{\pi}\lb[\int_{0}^{T} \la k_{1}^{i}(\cdot,s), k_{2}^{j}(\cdot,s)\ra \Gamma^{i,j} (s)\sqrt{\mu_{1}^{1}\mu_{2}^{1}}(\md s)\rb],
	\end{align*}
	where \(\Gamma^{i,j}\) is the density of \([M_{1}^{i},M_{2}^{j}]\) with respect to \(\sqrt{\mu_{1}^{1}\mu_{2}^{1}}\).
	By Kunita--Watanabe inequality, we derive
	\begin{equation*}
		\begin{pmatrix}
			\diag \lb(\frac{\md \mu_{1}^{1}}{\md \mu_{1}^{1}    },\dots, \frac{\md \mu_{1}^{m}}{\md \mu_{1}^{1}}\rb) & \Gamma                                                                                                   \\
			\Gamma^{*}                                                                                               & \diag \lb(\frac{\md \mu_{2}^{1}}{\md \mu_{2}^{1}    },\dots, \frac{\md \mu_{2}^{n}}{\md \mu_{2}^{1}}\rb)
		\end{pmatrix}(s)\geq 0.
	\end{equation*}
	By Lemma \ref{lem-tr}, we conclude the proof.
	In particular, the supremum is induced by the Gaussian coupling
	\[\left\{\begin{aligned}M_{1}(t) & =\int_{0}^{t}\sqrt{\diag \lb(\frac{\md \mu_{1}^{1}}{\md \nu    },\dots, \frac{\md \mu_{1}^{m}}{\md \nu}\rb)}(s)\md \tilde{M}(s),               \\
               M_{2}(t) & =\int_{0}^{t}\sqrt{ \diag \lb(\frac{\md \mu_{2}^{1}}{\md \nu    },\dots, \frac{\md \mu_{2}^{n}}{\md \nu}\rb)}(s)\Gamma^{*}(s)\md \tilde{M}(s),
		\end{aligned}\right.
	\]
	where \(\nu=\mu_{1}^{1}+\mu_{2}^{1}\), \(\Gamma\) a deterministic process attains the supremum in the above estimate,   \(\tilde{M}\) is a Gaussian martingale with independent increments and  \([\tilde{M}](\md s)=\Id \nu(\md s)\).


\end{proof}

\subsection{A martingale approximation to the fractional BMs}
It is well-known that, except in the case \(H=1/2\), the fractional Brownian motion is neither a martingale nor a Markov process.
Hence, models based on fractional Brownian motions in practice are often less tractable and lead to difficulty in numerical simulation.
To this end, we use the transfer principle  to derive the best martingale approximation of a fractional Brownian motion in terms of their adapted Wasserstein distance, i.e.,
\begin{equation}
	\label{eqn-mh}
	\inf_{M}\AW_{2}(B_{H},M)^{2}, \text{ where } M \text{ is a } \bF^{B_{H}}\text{-martingale}.
\end{equation}
\begin{thm}
	\label{thm-mart}
	Let \(k_{H}\) be the Molchan--Golosov kernel of the fractional Brownian motion \(B_{H}\).
	Then, the solution to \eqref{eqn-mh} is given by
	\begin{equation*}
		M_{H}(t)=\int_{0}^{t}  \frac{1}{T-r}\int_{r}^{T}k_{H}(s,r)\md s \md B(r).
	\end{equation*}
\end{thm}
\begin{proof}
	Since \(B_{H}(t)=\int_{0}^{t}k_{H}(t,s)\md B(s)\) is a canonical representation, we have \(\bF^{B_{H}}=\bF^{B}\).
	Without loss of generality, we may restrict \eqref{eqn-mh} to the set of centered and square integrable martingales.
	Under any bicausal coupling \(\pi\), \(M\) is still a \(\bF^{B}\)-martingale.
	By martingale representation theorem, we deduce
	\begin{equation*}
		M(t)=\int_{0}^{t} \rho(r) \md B(r), \text{ where } \rho \text{ is a } \bF^{B}\text{-predictable process}.
	\end{equation*}
	Therefore, we have
	\begin{align*}
		\inf_{M}\AW_{2}(B_{H},M)^{2} & =\inf_{\rho}\E\lb[\int_{0}^{T} \lb|B_{H}(s)-\int_{0}^{s} \rho(r) \md B(r)\rb|^{2}\md s\rb]     \\
		                             & =\inf_{\rho}\E\lb[\int_{0}^{T} \lb|\int_{0}^{s}(k_{H}(s,r)- \rho(r)) \md B(r)\rb|^{2}\md s\rb] \\
		                             & =\inf_{\rho}\E\lb[\int_{0}^{T} \int_{r}^{T} (k_{H}(s,r)-\rho(r))^{2}\md s \md r\rb].
	\end{align*}
	It is clear that the optimal \(\rho\) is given by \(\rho_{H}(r)=\frac{1}{T-r}\int_{r}^{T}k_{H}(s,r)\md s\).
\end{proof}
We can interpret \(M_{H}\) as the martingale whose volatility is given by the average volatility of the prediction process of \(B_{H}\).
To be more precise, we introduce the prediction process \(\Theta_{H}\) of \(B_{H}\) as the double-indexed process given by
\begin{equation*}
	\Theta_{H}(s;t):=\E[B_{H}(t)|\cF_{s}^{B_{H}}]=\int_{0}^{s}k_{H}(t,r)\md B(r) \text{ for } 0\leq s\leq t.
\end{equation*}
In particular, for any fixed \(t\in[0,T]\), \(\Theta_{H}(\cdot\,;t)\) is a martingale with volatility given by \(k_{H}(t,\cdot)\).
Therefore, the volatility of the martingale \(M_{H}\) at the current time \(r\), \(\rho_{H}(r)\),  is given by the current volatility  of the prediction process \(\Theta_{H}(\cdot\,;t)\) averaged over the future period \(t\in[r,T]\).

\section{Fractional SDEs}
\label{sec-fsde}
In this section, we investigate the adapted Wasserstein distance between 1D fractional SDEs.
Let \(X_{i}\) be the solution to
\begin{equation}
	\label{eqn-fsde}
	X_{i}(t)= x_{i} + \int_{0}^{t} b_{i}(X_{i}(s))\md s + \int_{0}^{t}\sigma_{i}(X_{i}(s ))\md Z_{i}(s),
\end{equation}
where \(Z_{i}(t)=\int_{0}^{t}k_{i}(t,s)\md B_{i}(s)\) and \(B_{i}\) is a standard Brownian motion.
\begin{asmp}
	\label{asmp-fsde}
	We assume
	\begin{itemize}
		\item \(b_{i},\sigma_{i}\in C^{2}\) with bounded first and second derivatives.
		\item \(b_{i}''\) and \(\sigma_{i}''\) are  uniform continuous with a modulus of continuity \(\varrho_{i}\).
		\item  \(\sigma_{i}\) is positive, bounded, and bounded away from 0.
	\end{itemize}
\end{asmp}

\begin{asmp}
	\label{asmp-k}
	We assume \(Z_{i}(t)=\int_{0}^{t}k_{i}(t,s)\md B_{i}(s)\) is a canonical representation.
	Moreover, \(k_{i}\) satisfies
	\begin{itemize}
		\item \(k_{i}(t,s)\geq 0\) for any \(t,s\in[0,T]\).
		\item \(k_{i}(\cdot,s)\in C^{1}([0,T];\bbR)\) for any \(s\in (0,T]\).
		\item \(
		      |k_{i}(t,s)|\leq C s^{1/2-H}|t-s|^{H-1/2}\) and \(|\partial_{t}k_{i}(t,s)|\leq C s^{1/2 -H}|t-s|^{H-3/2}\) for some \(H\in(1/2,1)\).
	\end{itemize}
\end{asmp}

\begin{asmp}
	\label{asmp-mono}
	We assume either of the following conditions holds:
	\begin{enumerate}[label=(\roman*)]
		\item \((b_{i}/\sigma_{i})\) is non-decreasing.
		\item \(k_{1}(\cdot,s)\) and \(k_{2}(\cdot,s)\) are both non-decreasing for any \(s\in(0,T]\).
	\end{enumerate}
\end{asmp}

The following well-posedness result is standard, and for example, can be found in \citet[Section 8.3]{friz20Course}, \citet[Theorem A.1]{viens19martingale}.
\begin{lem}
	\label{lem-prior}
	Under Assumptions \ref{asmp-fsde} and \ref{asmp-k}, fractional SDE \eqref{eqn-fsde}  is well-posed with a unique \(\alpha\)-H\"older continuous strong solution for any \(\alpha<H\).
	The stochastic integral \(\int_{0}^{t}\sigma_{i}(X_{i}(s ))\md Z_{i}(s)\) can be interpreted as a Young integral.
	Moreover, \(\E[\sup_{t\in\in[0,T]}|X_{i}(t)|^{p}]<\infty\) for any \(p\geq 1\).
\end{lem}

\begin{thm}
	\label{thm-fsde}
	Under Assumptions \ref{asmp-fsde}, \ref{asmp-k} and \ref{asmp-mono}, the adapted Wasserstein distance between \(X_{1}\) and \(X_{2}\)  is attained by the synchronous coupling between \(B_{1}\) and \(B_{2}\), i.e., the  noises \(Z_{1}\) and \(Z_{2}\) are driven by the same Brownian motion.
	In particular the synchronous coupling is a bicausal coupling between \(X_{1}\) and \(X_{2}\).
\end{thm}


\begin{rmk}
	Assumptions \ref{asmp-k} and \ref{asmp-mono} includes the Riemann--Liouville fractional kernel \(RL_{H}(t,s)=\Gamma(H+1/2)^{-1}(t-s)^{H-1/2}\1_{\{t\geq s\}}\), as well as the Molchan--Golosov kernel \(k_{H}(t,s)\)  for \(H\in(1/2,1)\).
\end{rmk}

We split the proof of Theorem \ref{thm-fsde} into two steps.
The first step is to show, by a stochastic control reformulation, the results hold for the additive noise, i.e., \(\sigma_{i}\equiv 1\).
In the second step, we apply Lamperti transform to reduce the general case to the additive noise case.

\subsection{Additive noise}
By strong well-posedness Lemma \ref{lem-prior}, we reduce the problem to a minimization over the bicausal coupling between the driving Brownian motions.
\begin{lem}
	Let \(\sigma_{i}\equiv 1\) for \(i=1,2\). Under Assumptions \ref{asmp-fsde} and \ref{asmp-k}, we have \(\Pi_{\bc}(X_{1},X_{2})=\Pi_{\bc}(B_{1},B_{2})\).
\end{lem}

\begin{proof}
	It suffices to show \(\bF^{X_{i}}=\bF^{B_{i}}\).
	From the strong well-posedness, we have \(\cF_{t}^{X_{i}}\subseteq \cF_{t}^{B_{i}}\) for any \(t\in[0,T]\).
	Moreover, we notice
	\begin{equation*}
		Z_{i}(t)=\int_{0}^{t}k_{i}(t,s)\md B_{i}(s)=X_{i}(t)-x_{i}-\int_{0}^{t}b_{i}(s,X_{i}(s))\md s\in \cF_{t}^{X_{i}}.
	\end{equation*}
	This implies \(\cF_{t}^{B_{i}}=\cF_{t}^{Z_{i}}\subseteq \cF_{t}^{X_{i}}\) from the canonical representation of \(Z_{i}\).
	Therefore, \(\cF_{t}^{X_{i}}=\cF_{t}^{B_{i}}\) and we conclude the proof.
\end{proof}

Now, similar to \citet{bion-nadal19Wassersteintype}, we address the bicausal optimal transport problem as a stochastic control problem with the control of the correlation of the driving Brownian motions.
We consider  a controlled system
\begin{equation*}
	\left\{
	\begin{aligned}
		X_{1}(t)     & = x_{1} + \int_{0}^{t} b_{1}(X_{1}(s))\md s + \int_{0}^{t}k_{1}(t,s)\md B_{1}(s),         \\
		X_{2}^{u}(t) & = x_{2} + \int_{0}^{t} b_{2}(X^{u}_{2}(s))\md s + \int_{0}^{t}k_{2}(t,s)\md B_{2}^{u}(s),
	\end{aligned}
	\right.
\end{equation*}
where \(\md B^{u}_{2}(t)=\sin(u(t))\md B_{1}(t)+ \cos(u(t))\md \tilde{B}_{1}(t)\) and \(\tilde{B}_{1}\) is a Brownian motion independent to \(B_{1}\).
We notice the control only enters the system through the correlation of the driving Brownian motions.
Our aim is to minimize
\begin{equation*}
	\inf_{u\in\scrU_{[0,T]}}\E\lb[\int_{0}^{T} \lb|X_{1}(t)-X_{2}^{u}(t)\rb|^{2}\md t\rb],
\end{equation*}
over \(\scrU_{[0,T]}\)  the set of \((\bF^{B_{1}}\vee \bF^{B^u_2})\)-predictable processes.
We immediately see that \(X_{1}\) no longer enjoys the flow property in the sense that
\begin{equation*}
	X_{1}(t)\neq \tilde{X}_{1}^{s,X_{1}}(t), \quad \text{ where }\quad \tilde{X}_{1}^{s,X_{1}}(t):= X_{1}(s)+\int_{s}^{t} b(X_{1}(r))\md r +\int_{s}^{t}k_{1}(t,s)\md B_{1}(r).
\end{equation*}
Therefore, the classical approach of dynamic programming does not apply directly.
To go around this issue an auxiliary system \(\Theta\) is introduced in \citet{viens19martingale} to retrieve the flow property.
We adapt their framework to our setting as
\begin{equation*}
	\left\{
	\begin{aligned}
		\Theta_{1}(s;t)     & =x_{1}+ \int_{0}^{s}b_{1}(\Theta_{1}(r;r))\md r + \int_{0}^{s} k_{1}(t,r)\md B_{1}(r),         \\
		\Theta_{2}^{u}(s;t) & = x_{2} + \int_{0}^{s}b_{2}(\Theta_{2}^{u}(r;r))\md r +\int_{0}^{s}k_{2}(t,r)\md B_{2}^{u}(r).
	\end{aligned}
	\right.
\end{equation*}
In particular, \((X_{1}(t),X_{2}^{u}(t))=(\Theta_{1}(t;t),\Theta_{2}^{u}(t;t))\) and
\begin{equation*}
	\AW_{2}(X_{1},X_{2} )^{2}=\inf_{\pi\in\Pi_{\bc}(B_{1},B_{2})}\E_{\pi}[\|X_{1}-X_{2}\|^{2}]=\inf_{u\in\scrU_{[0,T]}}\E\lb[\int_{0}^{T} \lb|\Theta_{1}(t;t)-\Theta_{2}^{u}(t;t)\rb|^{2}\md t\rb].
\end{equation*}
We can view \(\{(\Theta_{1}(t;\cdot),\Theta_{2}^{u}(t;\cdot)):t\in[0,T]\}\) as an infinite dimensional flow taking values in  \(C([0,T];\bbR^{2})\).
Naturally, we define the value function \(v:[0,T]\times C([0,T];\bbR^{2})\to\bbR  \) as
\begin{equation*}
	v(r,\omega_{1}, \omega_{2}):=\inf_{u\in \scrU_{[r,T]}} \E\lb[\int_{r}^{T} \lb|\Theta_{1}^{r,\omega_{1}}(t;t)-\Theta_{2}^{r,\omega_{2},u}(t;t)\rb|^{2}\md t \rb],
\end{equation*}
where
\begin{equation*}
	\left\{
	\begin{aligned}
		\Theta_{1}^{r,\omega_{1}}(\cdot\,;t)   & =\omega_{1}(t)+ \int_{r}^{\cdot}b_{1}(\Theta_{1}^{r,\omega_{1}}(s;s))\md s + \int_{r}^{\cdot}k_{1}(t,s)\md B_{1}(s),         \\
		\Theta_{2}^{r,\omega_{2},u}(\cdot\,;t) & = \omega_{2}(t) + \int_{r}^{\cdot}b_{2}(\Theta_{2}^{r,\omega_{2},u}(s;s))\md s +\int_{r}^{\cdot }k_{2}(t,s)\md B_{2}^{u}(s).
	\end{aligned}
	\right.
\end{equation*}
We denote the time derivative by \(\partial_{t}\) and the first and second Fr\'echet derivatives  by \(\partial_{\omega_{i}}\) and \(\partial^{2}_{\omega_{i}\omega_{j}}\), respectively.
The corresponding HJB equation is given by
\begin{equation}
	\label{eqn-hjb}
	(\partial_{t} + \cL_{1} +\cL_{2} + \cH) V(r,\omega_{1},\omega_{2})=-|\omega_{1}(r)-\omega_{2}(r)|^{2} \quad \text{ with }\quad V(T,\omega_{1},\omega_{2})=0.
\end{equation}
Here, \(\cL_{i}\) and \(\cH\) are given by
\begin{equation*}
	\cL_{i}V(r,\omega_{1},\omega_{2})=\la b_{i}( \omega_{i}(r))\1_{[0,T]}(\cdot),\partial_{\omega_{i}}V(r,\omega_{1},\omega_{2})(\cdot) \ra +\frac{1}{2}\la (k_{1}(\cdot,r),k_{1}(\cdot,r)) , \partial^{2}_{\omega_{i}\omega_{i}}V(r,\omega_{1},\omega_{2})(\cdot)\ra
\end{equation*}
and
\begin{equation*}
	\cH V(r,\omega_{1},\omega_{2}) = \inf_{a\in[-1,1]} a\la (k_{1}(\cdot,r),k_{2}(\cdot,r)), \partial^{2}_{\omega_{1}\omega_{2}}V(r,\omega_{1},\omega_{2})(\cdot)\ra.
\end{equation*}

We denote the expected cost under the synchronous coupling by \(V_{*}\), which is given by
\begin{equation*}
	V_{*}(r,\omega_{1}, \omega_{2}):= \E\lb[\int_{r}^{T} \lb|\Theta_{1,*}^{r,\omega_{1}}(t;t)-\Theta_{2,*}^{r,\omega_{2}}(t;t)\rb|^{2}\md t \rb],
\end{equation*}
where
\begin{equation*}
	\left\{
	\begin{aligned}
		\Theta_{1,*}^{r,\omega_{1}}(\cdot\,;t) & =\omega_{1}(t)+ \int_{r}^{\cdot}b_{1}(\Theta_{1,*}^{r,\omega_{1}}(s;s))\md s + \int_{r}^{\cdot}k_{1}(t,s) \md B_{1}(s),   \\
		\Theta_{2,*}^{r,\omega_{2}}(\cdot\,;t) & = \omega_{2}(t) + \int_{r}^{\cdot}b_{2}(\Theta_{2,*}^{r,\omega_{2}}(s;s))\md s +\int_{r}^{\cdot } k_{2}(t,s)\md B_{1}(s).
	\end{aligned}
	\right.
\end{equation*}

Our plan is to verify \(V_{*}\) coincides with the value function \(v\), which requires us to adapt a functional It\^o formula from \citet[Theorem 3.10]{viens19martingale} to our setting.
We remark that the kernel \(k_{i}\) here has the same singularity as the fractional Brownian motion kernel \(k_{H}(t,s)\) for \(H\in(1/2,1)\), and the singularity can only occur when \(s\) approaches 0.
Hence, for any \(r>0\) \citet[Theorem 3.10]{viens19martingale} is directly applicable.
This result is probabilistic which should be contrasted with the functional It\^o formula developed in \citet{dupire09Functional,cont10Change} where the authors derived a \emph{pathwise} It\^o formula for non-anticipative functionals.
\begin{lem}[Functional It\^o formula]
	\label{lem-ito}
	Let \(u:[0,T]\times C([0,T];\bbR^{2})\to \bbR\) be a purely anticipative functional, i.e., \(u(t,\omega_{1},\omega_{2})=u(t,\omega_{1}(\cdot\vee t),\omega_{2}(\cdot \vee t))\) for any \(t\in[0,T]\) and \(\omega_{i}\in C([0,T];\bbR)\).
	Assume \(u\in C^{1,2}\), and there exists a modulus of continuity \(\rho\) such that for  any  \(\eta , \tilde{\eta}\in C([0,T];\bbR)\),  \(u\) satisfies the following conditions:
	\begin{enumerate}[label=(\roman*)]
		\item for any \(\omega_{1},\omega_{2}\in C([0,T];\bbR)\), \begin{equation*}
			      \begin{aligned}
				      |\la \eta, \partial_{\omega_{i}}u(r,\omega_{1},\omega_{2})\ra|                              & \leq C (1+\|\omega_{1}\|_{\infty}+\|\omega_{2}\|_{\infty}) \|\eta\|_{\infty}                           \\
				      |\la (\eta,\tilde{\eta}), \partial^{2}_{\omega_{i}\omega_{j}}u(r,\omega_{1},\omega_{2})\ra| & \leq C (1+\|\omega_{1}\|_{\infty}+\|\omega_{2}\|_{\infty}) \|\eta\|_{\infty}\|\tilde{\eta}\|_{\infty};
			      \end{aligned}
		      \end{equation*}
		\item for any other \(\omega_{1}',\omega_{2}'\in C([0,T];\bbR)\),
		      \begin{equation*}
			      \begin{aligned}
				       & |\la (\eta,\tilde{\eta}), \partial^{2}_{\omega_{i}\omega_{j}}u(r,\omega_{1},\omega_{2})-\partial^{2}_{\omega_{i}\omega_{j}}u(r,\omega_{1}',\omega_{2}')\ra|                                     \\
				       & \hspace{3cm}\leq C (1+\|\omega_{1}\|_{\infty}+\|\omega_{2}\|_{\infty}) \|\eta\|_{\infty}\|\tilde{\eta}\|_{\infty}\rho(\|\omega_{1}-\omega_{1}'\|_{\infty}+\|\omega_{2}-\omega_{2}'\|_{\infty}).
			      \end{aligned}
		      \end{equation*}
		      Then under Assumptions \ref{asmp-fsde} and \ref{asmp-k}, we have
		      \begin{align*}
			      u(t,\Theta_{1,*}^{r,\omega_{1}}(t;\cdot),\Theta_{2,*}^{r,\omega_{2}}(t;\cdot)) & =u(r,\omega_{1},\omega_{2})+\int_{r}^{t} (\partial_{t}+\cL_{1}+\cL_{2})u(s,\Theta_{1,*}^{r,\omega_{1}}(s;\cdot),\Theta_{2,*}^{r,\omega_{2}}(s;\cdot))\md s                            \\
			                                                                                     & \qquad + \int_{r}^{t}\la (k_{1}(\cdot,s),k_{2}(\cdot,s)), \partial_{\omega_{1}\omega_{2}}^{2}u(s, \Theta_{1,*}^{r,\omega_{1}}(s;\cdot),\Theta_{2,*}^{r,\omega_{2}}(s;\cdot))\ra \md s \\
			                                                                                     & \qquad + \int_{r}^{t}\la k_{1}(\cdot,s),\partial_{\omega_{1}}u(s, \Theta_{1,*}^{r,\omega_{1}}(s;\cdot),\Theta_{2,*}^{r,\omega_{2}}(s;\cdot))\ra \md B_{1}(s)                          \\
			                                                                                     & \qquad + \int_{r}^{t}\la k_{2}(\cdot,s),\partial_{\omega_{2}}u(s, \Theta_{1,*}^{r,\omega_{1}}(s;\cdot),\Theta_{2,*}^{r,\omega_{2}}(s;\cdot))\ra \md B_{1}(s).
		      \end{align*}
	\end{enumerate}

\end{lem}

The following technical lemma states that \(V_{*}\) is sufficiently regular to apply the functional It\^o formula of \citet{viens19martingale}.
\begin{lem}
	\label{lem-reg}
	Under Assumptions \ref{asmp-fsde}, \ref{asmp-k}, and \ref{asmp-mono}, \(V_{*}\) satisfies conditions in Lemma~\ref{lem-ito}, and is a classical solution to
	\begin{equation}
		\label{eqn-v*}
		(\partial_{t}+\cL_{1}+\cL_{2})V_{*}(r,\omega_{1},\omega_{2})+\la (k_{1}(\cdot,r),k_{2}(\cdot,r)), \partial^{2}_{\omega_{1}\omega_{2}}V_{*}(r,\omega_{1},\omega_{2})(\cdot)\ra = -|\omega_{1}(r)-\omega_{2}(r)|^{2}.
	\end{equation}
	Moreover, there is a probabilistic representation of \(\partial^{2}_{\omega_{1}\omega_{2}}V_{*}\) given by
	\begin{equation}
		\label{eqn-cross}
		\la (\eta_{1},\eta_{2}), \partial_{\omega_{1}\omega_{2}}^{2}V_{*}(r,\omega_{1},\omega_{2})\ra=-2E\lb[\int_{r}^{T}\la \eta_{1}, \Gamma_{1,*}^{r,\omega_{1}}(t)\ra\la \eta_{2},\Gamma_{2,*}^{r,\omega_{2}}(t)\ra\md t\rb],
	\end{equation}
	where \(\Gamma_{i,*}^{r,\omega_{i}}\) is the unique solution to
	\begin{equation}
		\label{eqn-gam}
		\Gamma_{i,*}^{r,\omega_{i}}(t)=\delta(t)+ \int_{r}^{t} b_{i}'(\Theta_{i,*}^{r,\omega_{i}}(s;s)) \Gamma_{i,*}^{r,\omega_{i}}(s)\md s.
	\end{equation}
\end{lem}
To not distract the readers, we postpone the proof of this technical result to Section \ref{sec-estimate} and continue with the main line of our results.

\begin{thm}
	Under Assumptions \ref{asmp-fsde}, \ref{asmp-k}, and \ref{asmp-mono}, \(V_{*}\) is a classical solution to the path-dependent HJB equation \eqref{eqn-hjb}.
	Moreover, \(V_{*}\) coincides with the value function \(v\), and in particular, the adapted Wasserstein distance is given by \(\AW_{2}(X_{1},X_{2})= V_{*}(0,x_{1}\1_{[0,T]},x_{2}\1_{[0,T]})^{1/2}\).
\end{thm}

\begin{rmk}
	We point out that a similar stochastic control approach was taken in \citet{bion-nadal19Wassersteintype} where they rely on the regularity and well-posedness of nonlinear parabolic equations.
	However, to the best of knowledge, there is no well-posedness result for nonlinear \emph{functional} parabolic equations on Banach space which can be directly applied to our setting.
	Our estimates are based on probabilistic methods.
	We manage to show the existence of the classical solution to the path-dependent HJB equation by a direct construction.
	It is interesting and challenging to build a viscosity solution theory of this type of path-dependent HJB equations.
	We leave this as a future research direction.
\end{rmk}
\begin{rmk}
	\label{rmk-gen}
	Following the same line of proof, we can show that for any non-decreasing \(f_{i}\) with bounded first, second, and third derivatives, synchronous coupling is still an optimal coupling for the bicausal optimal transport problem
	\begin{equation*}
		\inf_{\pi\in \Pi_{\bc}(X_{1},X_{2})}\E_{\pi}\lb[\int_{0}^{T}|f_{1}(X_{1}(t))-f_{2}(X_{2}(t))|^{2}\md t\rb].
	\end{equation*}
	Also, see Remark \ref{rmk-c} for more details.
\end{rmk}

\begin{proof}
	We prove  \(V_{*}\) is a classical solution to the HJB equation \eqref{eqn-hjb}.
	By Lemma \ref{lem-reg}, it suffices to verify that \(\la (k_{1}(\cdot,r),k_{2}(\cdot,r)),\partial_{\omega_{1}\omega_{2}}^{2}V_{*}(r,\omega_{1},\omega_{2})\ra\leq ~0\).
	Recall we define \(\Gamma^{r,\omega_{i}}_{i,*}\) in \eqref{eqn-gam}, and it admits a unique solution
	\begin{equation}
		\label{eqn-gam2}
		\Gamma_{i,*}^{r,\omega_{i}}(t)=\delta(t)+ \int_{r}^{t}\exp \lb(\int_{s}^{t}b_{i}'(\Theta_{i,*}^{r,\omega_{i}}(\tau;\tau))\md \tau  \rb) b_{i}'(\Theta_{i,*}^{r,\omega_{i}}(s;s)) \delta(s)\md s.
	\end{equation}
	We discuss two cases in Assumption \ref{asmp-mono} separately.
	If \(b_{i}\) is non-decreasing, from \eqref{eqn-gam2}, we derive \(\la \eta_{i}, \Gamma^{r,\omega_{i} }_{i}\ra\geq 0\) for any \(\eta_{i}\geq 0\).
	Plugging it into \eqref{eqn-cross}, we conclude \(V_{*}\) is a classical solution to HJB equation \eqref{eqn-hjb} as \(k_{i}(\cdot,r)\geq  0\).
	If \(k_{1}(\cdot,r)\) and \(k_{2}(\cdot,r)\) are both non-decreasing, by applying integration by part to \eqref{eqn-gam2}, we derive
	\begin{equation*}
		\la k_{i}(\cdot, r), \Gamma_{i,*}^{r,\omega_{i}}(t)\ra = \int_{r}^{t} \exp\lb(\int_{s}^{t}b_{i}'(\Theta_{i,*}^{r,\omega_{i}}(\tau;\tau))\md \tau\rb) k_{i}(\md s,r)
	\end{equation*}
	have the same sign for \(i=1,2\).
	Therefore, \(V_{*}\) is a classical solution to \eqref{eqn-hjb}.

	We show that \(V_{*}\) coincides with the value function \(v\).
	We fix a control \(u\in \scrU_{[r,T]}\) and, by Lemma~\ref{lem-reg}, we apply functional It\^{o} formula to \(V_{*}(t, \Theta_{1}^{r,\omega_{1}}(t;\cdot),\Theta_{2}^{r,\omega_{2},u}(t;\cdot))\).
	We obtain
	\begin{align*}
		 & \quad V_{*}(r,\omega_{1},\omega_{2})                                                                                                                                                                   \\
		 & = -\E\lb[\int_{r}^{T} (\partial_{t}+\cL_{1}+\cL_{2})V_{*}(t, \Theta_{1}^{r,\omega_{1}}(t;\cdot),\Theta_{2}^{r,\omega_{2},u}(t;\cdot))    \md t \rb]                                                    \\
		 & \quad- \E\lb[\int_{r}^{T}\sin(u(t))\la (k_{1}(t,\cdot),k_{2}(t,\cdot)), \partial_{\omega_{1}\omega_{2}}V_{*}(t, \Theta_{1}^{r,\omega_{1}}(t;\cdot),\Theta_{2}^{r,\omega_{2},u}(t;\cdot))\ra \md t \rb] \\
		 & \leq  -\E\lb[\int_{r}^{T} (\partial_{t}+\cL_{1}+\cL_{2}+\cH)V_{*}(t, \Theta_{1}^{r,\omega_{1}}(t;\cdot),\Theta_{2}^{r,\omega_{2},u}(t;\cdot))    \md t \rb]                                            \\
		 & = \E\lb[\int_{r}^{T}\lb|\Theta_{1}^{r,\omega_{1}}(t;t)-\Theta_{2}^{r,\omega_{2},u}(t;t)\rb|^{2}\md t\rb].
	\end{align*}
	The above inequality follows from the fact that \(V_{*}\) satisfies HJB equation \eqref{eqn-hjb}.
	Therefore, taking infimum over \(\scrU_{[r,T]}\) we deduce
	\begin{equation*}
		V_{*}(r,\omega_{1},\omega_{2})\leq \inf_{u\in\scrU_{[r,T]}} \E\lb[\int_{r}^{T}\lb|\Theta_{1}^{s,\omega_{1}(t)}(t;t)-\Theta_{2}^{s,\omega_{2}(t),u}(t;t)\rb|^{2}\md t\rb]=v(r,\omega_{1},\omega_{2}).
	\end{equation*}
	On the other hand, we notice \(u(r)\equiv \pi/2\) gives an optimal control, and hence \(V_{*}=v\).

\end{proof}

\subsection{Multiplicative noise}
Now we return to \eqref{eqn-fsde} with diffusion coefficient \(\sigma_{i}\) satisfying Assumption \ref{asmp-fsde}.
We write
\begin{equation*}
	g_{i}(x)=\int_{x_{i}}^{x}\frac{1}{\sigma_{i}(\xi)}\md \xi \quad \text{ and } \quad Y_{i}(t)= g_{i}(X_{i}(t)).
\end{equation*}
Notice that under Assumptions \ref{asmp-fsde} and \ref{asmp-k}, \(X_{i}\) and \(Z_{i}\) are \(\alpha\)--H\"older with \(\alpha>1/2\).
This yields, \(Y_{i}\), the Lamperti transform of \(X_{i}\), satisfies
\begin{equation*}
	Y_{i}(t)=\int_{0}^{t} \frac{b_{i}(g_{i}^{-1}(Y_{i}(s)))}{\sigma_{i}(g_{i}^{-1}(Y_{i}(s)))}\md s +Z_{i}(t).
\end{equation*}

\begin{lem}
	\label{lem-mn}
	Under Assumptions \ref{asmp-fsde} and \ref{asmp-k}, we have \(\bF^{X_{i}}=\bF^{B_{i}}\).
\end{lem}
\begin{proof}
	By Lemma \ref{lem-prior}, \(X_{i}\) is a strong solution and hence \(\cF_{t}^{X_{i}}\subseteq \cF_{t}^{B_{i}}\) for any \(t\in[0,T]\).
	On the other hand, we notice \(Z_{i}(t)= Y_{i}(t)-\int_{0}^{t} \frac{b_{i}(g_{i}^{-1}(Y_{i}(s)))}{\sigma_{i}(g_{i}^{-1}(Y_{i}(s)))}\md s\), which implies \(\cF_{t}^{Z_{i}}\subseteq \cF_{t}^{Y_{i}}\).
	Therefore, we deduce
	\begin{equation*}
		\cF_{t}^{B_{i}}=\cF_{t}^{Z_{i}}\subseteq \cF_{t}^{Y_{i}}\subseteq \cF_{t}^{X_{i}}\subseteq \cF_{t}^{B_{i}}.
	\end{equation*}
\end{proof}

The above lemma allows us to reduce the adapted Wasserstein distance between \(X_{1}\) and \(X_{2}\) to a bicausal optimal transport problem between \(Y_{1}\) and \(Y_{2}\) as
\begin{equation*}
	\AW_{2}(X_{1},X_{2})^{2}=\inf_{\pi\in\Pi_{\bc}(Y_{1},Y_{2})}\E\lb[\int_{0}^{T} |g_{1}^{-1}(Y_{1}(t))-g_{2}^{-1}(Y_{2}(t))|^{2}\md t\rb].
\end{equation*}

We construct \((\tilde{b}_{i},\tilde{\sigma}_{i})=\lb(\frac{b_{i}\circ g_{i}^{-1}}{\sigma_{i}\circ g_{i}^{-1}},1\rb)\).
A direct calculation gives \(\tilde{b}_{i}'=b_{i}'\circ g_{i}^{-1}-\frac{(b_{i}\circ g_{i}^{-1})( \sigma_{i}' \circ g_{i}^{-1})}{\sigma_{i}\circ g_{i}^{-1}}\) and \[\tilde{b}_{i}''= (b_{i}''\circ g_{i}^{-1})(\sigma\circ g_{i}^{-1}) -(b_{i}'\circ g_{i}^{-1})( \sigma_{i}' \circ g_{i}^{-1}) -(b_{i}\circ g_{i}^{-1})( \sigma_{i}'' \circ g_{i}^{-1}) + \frac{(b_{i}\circ g_{i}^{-1})( \sigma_{i}' \circ g_{i}^{-1})^{2}}{\sigma_{i}\circ g_{i}^{-1}}.\]
If \(b_{i}\) were bounded,  we could verify \((\tilde{b}_{i},\tilde{\sigma}_{i})\) satisfies Assumptions \ref{asmp-fsde}, and \((\tilde{b_{i}}/\tilde{\sigma_{i}})\) is non-decreasing if \((b_{i}/\sigma_{i})\) is.
Applying Remark \ref{rmk-gen} we could conclude the proof of Theorem \ref{thm-fsde}.
For unbounded \(b_{i}\), we take a sequence of functions \(b^{n}_{i}\in C^{2}_{b}\)  satisfying Assumption \ref{asmp-fsde} and converging to \(b_{i}\)  pointwise.
In particular, we can assume \(b_{i}^{n}=b_{i}\) on \([-n,n]\), and \(|(b^{n}_{i})'|\leq |b_{i}'|\leq L\).
We define
\begin{equation*}
	X^{n}_{i}(t)= x_{i} + \int_{0}^{t} b^{n}_{i}(X^{n}_{i}(s))\md s + \int_{0}^{t}\sigma_{i}(X^{n}_{i}(s))\md Z_{i}(s).
\end{equation*}
By the triangle inequality, we obtain
\begin{equation*}
	\AW_{2}(X_{1}^{n},X_{2}^{n})\leq \AW_{2}(X_{1},X_{2})+ \AW_{2}(X_{1}^{n},X_{1})+ \AW_{2}(X_{2}^{n},X_{2}).
\end{equation*}
In order to show the synchronous coupling is optimal, we only need to show \(\AW_{2}(X_{i}^{n},X_{i})\) goes to 0 since the synchronous coupling is already optimal between \(X^{n}_{1}\) and \(X^{n}_{2}\) by previous arguments.

\begin{lem}
	Under Assumptions \ref{asmp-fsde} and \ref{asmp-k}, we have \(\lim_{n\to\infty}\AW_{2}(X_{i}^{n},X_{i})= 0\).
\end{lem}
\begin{proof}

	By Lemma \ref{lem-mn}, we have \(\bF^{X_{i}^{n}}=\bF^{Z_{i}^{n}}\), and hence the synchronous coupling \(\pi_{\sync}\) between  \(Z_{i}^{n}\) and \(Z_{i}\) is a bicausal coupling between \(X^{n}_{i}\) and \(X_{i}\).
	We write \(Y^{n}_{i}= g_{i}(X^{n}_{i})\) and \(\tilde{b}^{n}_{i}=\frac{b^{n}_{i}\circ g_{i}^{-1}}{ \sigma_{i}\circ g_{i}^{-1}}\).
	By our construction of \(b^{n}_{i}\), we have \(\tilde{b}^{n}_{i}=\tilde{b}_{i}\) on \([-n,n]\) and \(|(b^{n}_{i})'|\leq |b_{i}'|\leq L\).
	Without loss of generality, we may assume \(|b^{n}_{i}(x)|+|b_{i}(x)|\leq L(1+|x|)\) for possibly larger \(L\).

	Since \(\sigma_{i}\) is bounded and bounded away from 0, we derive that
	\begin{equation*}
		\AW_{2}(X_{i}^{n},X_{i})^{2}\leq \E_{\pi_{\sync}}[\|X_{i}^{n}-X_{i}\|^{2}=\E_{\pi_{\sync}}[\|g_{i}^{-1}(Y_{i}^{n})-g_{i}^{-1}(Y_{i})\|^{2}]\leq C  \E_{\pi_{\sync}}[\|Y_{i}^{n}-Y_{i}\|^{2}].
	\end{equation*}
	Therefore, it suffices to show \(Y^{n}_{i}\) converges to \(Y_{i}\) in \(H\)  in \(L^{2}\).
	Notice under \(\pi_{\sync}\), we have
	\begin{align*}
		|Y^{n}_{i}(t)-Y_{i}(t)|^{2} & \leq 2\lb(\int_{0}^{t} |\tilde{b}_{i}^{n}(Y_{i}^{n}(s)) -\tilde{b}_{i}^{n}(Y_{i}(s))|\md s\rb)^{2}+2\lb(\int_{0}^{t} |\tilde{b}_{i}^{n}(Y_{i}(s))-\tilde{b}_{i}(Y_{i}(s))|\md s\rb)^{2} \\
		                            & \leq 2TL^{2} \int_{0}^{t} |Y^{n}_{i}(s)-Y_{i}(s)|^{2}\md s + 2T\int_{0}^{t} |\tilde{b}^{n}_{i}(Y_{i}(s))-\tilde{b}_{i}(Y_{i}(s))|^{2}\1_{\{|Y_{i} (s)|\geq n\}}\md s                    \\
		                            & \leq 2TL^{2}\int_{0}^{t} |Y^{n}_{i}(s)-Y_{i}(s)|\md s + 2TL^{2}\int_{0}^{t}(1+|Y_{i}(s)|)^{2} \1_{\{|Y_{i} (s)|\geq n\}}\md s.
	\end{align*}
	By Gronwall's inequality, we obtain
	\begin{equation*}
		\E_{\pi_{\sync}}[\|Y_{i}^{n}-Y_{i}\|^{2}]\leq C\int_{0}^{T}\E_{\pi_{\sync}}\lb[|Y_{i}(s)|^{2}\1_{\{|Y_{i} (s)|\geq n\}}\rb]\md s.
	\end{equation*}
	By Lemma \ref{lem-prior}, \(Y_{i}\) is in \(L^{2}\), and hence we derive the \(L^{2}\) convergence of \(Y_{i}^{n}\).
\end{proof}

\section{Some additional estimates}
\label{sec-estimate}
Recall
\begin{equation}
	\label{eqn-theta-ode}
	\Theta_{i,*}^{r,\omega_{i}}(\cdot\,;t)=\omega_{i}(t)+ \int_{r}^{\cdot}b_{i}(\Theta_{i,*}^{r,\omega_{i}}(s;s))\md s + \int_{r}^{\cdot}k_{i}(t,s) \md B_{1}(s),
\end{equation}
and
\begin{equation}
	\label{eqn-gamma-ode}
	\Gamma_{i,*}^{r,\omega_{i}}(t)=\delta(t)+ \int_{r}^{t} b_{i}'(\Theta_{i,*}^{r,\omega_{i}}(s;s)) \Gamma_{i,*}^{r,\omega_{i}}(s)\md s.
\end{equation}
\begin{prop}
	\label{prop-theta}
	Let \(s\in[r,T]\) and \(\eta\in C([0,T];\bbR)\). Under Assumptions \ref{asmp-fsde} and \ref{asmp-k}, the following estimates hold with a deterministic constant \(C\) independent of \(\omega_{i}\) and \(\eta\)
	\begin{equation*}
		\sup_{t\in[r,T]}\E[|\Theta_{i,*}^{r,\omega_{i}}(t;t)|]\leq C(1+\|\omega_{i}\|_{\infty}) \text{ and } \sup_{t\in[r,T]}|\Theta^{r,\omega_{i}+\eta}_{i,*}(t;t)-\Theta_{i,*}^{r,\omega_{i}}(t;t)| \leq C  \|\eta\|_{\infty}.
	\end{equation*}
\end{prop}
\begin{proof}
	It follows directly from the Gronwall inequality and the boundedness of \(b'\).
\end{proof}

\begin{prop}
	\label{prop-gamma}
	Let \(s\in[r,T]\) and \(\tilde{\eta},\eta\in C([0,T];\bbR)\). Under Assumptions \ref{asmp-fsde} and \ref{asmp-k}, the following estimates hold with a deterministic constant \(C\) independent of \(\omega_{i}\), \(\eta\), and \(\tilde{\eta}\)
	\begin{equation*}
		\sup_{t\in[r,T]}|\la \tilde{\eta},\Gamma_{i,*}^{r,\omega_{i}}(t)\ra|\leq C \|\tilde{\eta}\|_{\infty}\text{ and } \sup_{t\in[r,T]}|\la \tilde \eta, \Gamma_{i,*}^{r,\omega_{i}+\eta}(t)-\Gamma_{i,*}^{r,\omega_{i}}(t)\ra| \leq C  \|\tilde{\eta}\|_{\infty}\|\eta\|_{\infty}.
	\end{equation*}
\end{prop}
\begin{proof}
	It follows directly from the Gronwall inequality and the boundedness of \(b'\) and \(b''\).
\end{proof}
The following result shows that \(\Gamma_{i,*}^{r,\omega_{i}}\) is the first variation process of \(\Theta_{i,*}^{r,\omega_{i}}\).
\begin{prop}
	\label{prop-1}
	Let \(\eta\in C([0,T];\bbR)\).
	Under Assumptions \ref{asmp-fsde} and \ref{asmp-k}, there exists a deterministic constant \(C\) independent of \(\omega_{i}\) and \(\eta\) such that \[\sup_{t\in[r,T]}|\Theta_{i,*}^{r,\omega_{i}+\eta}(t;t)-\Theta_{i,*}^{r,\omega_{i}}(t;t)-\la \eta, \Gamma_{i,*}^{r,\omega_{i}}(t)\ra |\leq C\|\eta\|_{\infty}^{2}.\]
\end{prop}
\begin{proof}
	Write \(\Delta \Theta(t)=\Theta_{i,*}^{r,\omega_{i}+\eta}(t;t)-\Theta_{i,*}^{r,\omega_{i}}(t;t)\) and \(R_{1}(t)= \Delta\Theta(t)- \la \eta, \Gamma_{i,*}^{r,\omega_{i}}(t)\ra\).
	Plugging \eqref{eqn-theta-ode} and \eqref{eqn-gamma-ode}, we notice that
	\begin{align*}
		R_{1}(t)= & \int_{r}^{t}b_{i}'(\Theta_{i,*}^{r,\omega_{i}}(s;s))R_{1}(s)\md s                                                                                                                        \\
		          & + \int_{r}^{t} \lb(\int_{0}^{1}\lb[b_{i}'(\Theta_{i,*}^{r,\omega_{i}}(s;s) +\lambda \Delta \Theta(s))-b_{i}' (\Theta_{i,*}^{r,\omega_{i}}(s;s))\rb]\md \lambda \rb)\Delta\Theta(s)\md s.
	\end{align*}
	By Gronwall  inequality and Proposition \ref{prop-theta}, we deduce
	\begin{equation*}
		\sup_{t\in[r,T]}|R_{1}(t)|\leq C  \|\eta\|_{\infty} \int_{r}^{T} \int_{0}^{1}\lb|b_{i}'(\Theta_{i,*}^{r,\omega_{i}}(s;s) +\lambda \Delta \Theta(s))-b_{i}' (\Theta_{i,*}^{r,\omega_{i}}(s;s))\rb|\md \lambda \md s .
	\end{equation*}
	Since \(b_{i}''\) is bounded and \(\sup_{t\in[r,T]}|\Delta\Theta(t)|\leq C \|\eta\|_{\infty}\), we derive \(\sup_{t\in[r,T]}|R_{1}(t)|\leq C \|\eta\|_{\infty}^{2}\).
\end{proof}

We define
\begin{equation}
	\label{eqn-xi}
	\Xi_{i,*}^{r,\omega_{i}}(t)=\int_{r}^{t}\exp \lb(\int_{s}^{t}b_{i}'(\Theta_{i,*}^{r,\omega_{i}}(\tau;\tau))\md \tau  \rb) b_{i}''(\Theta_{i,*}^{r,\omega_{i}}(s;s))\Gamma_{i,*}^{r,\omega_{i}}(s)\otimes\Gamma_{i,*}^{r,\omega_{i}}(s) \md s,
\end{equation}
which is the unique solution to
\begin{equation}
	\label{eqn-xi-ode}
	\Xi_{i,*}^{r,\omega_{i}}(t)=\int_{r}^{t}b_{i}'(\Theta_{i,*}^{r,\omega_{i}}(s;s))\Xi_{i,*}^{r,\omega_{i}}(s)\md s +\int_{r}^{t}b_{i}''(\Theta_{i,*}^{r,\omega_{i}}(s;s))\Gamma_{i,*}^{r,\omega_{i}}(s)\otimes\Gamma_{i,*}^{r,\omega_{i}}(s)\md s.
\end{equation}
\begin{prop}
	\label{prop-2}
	Let \(\eta,\tilde{\eta}\in C([0,T];\bbR)\).
	Under Assumptions \ref{asmp-fsde} and \ref{asmp-k}, there exists a deterministic constant \(C\) independent of \(\omega_{i}\), \(\eta\), and \(\tilde{\eta}\) such that
	\[ \sup_{t\in[r,T]}|\la \tilde{\eta}, \Gamma_{i,*}^{r,\omega_{i}+\eta}(t)- \Gamma_{i,*}^{r,\omega_{i}}(t)\ra-\la (\eta,\tilde{\eta}), \Xi_{i,*}^{r,\omega_{i}}(t)\ra|\leq   C\|\tilde{\eta}\|_{\infty}\|\eta\|_{\infty}\varrho_{i}(\|\eta\|_{\infty}),\]
	where \(\varrho_{i}\) is the modulus of continuity of \(b_{i}''\).
\end{prop}

\begin{proof}
	Write \(\Delta \Gamma(t)=\la \tilde{\eta}, \Gamma_{i,*}^{r,\omega_{i}+\eta}(t)-  \Gamma_{i,*}^{r,\omega_{i}}(t)\ra \) and  \(R_{2}(t)= \Delta \Gamma(t)-\la (\eta,\tilde{\eta}), \Xi_{i,*}^{r,\omega_{i}}(t)\ra \).
	Plugging \eqref{eqn-gamma-ode} and \eqref{eqn-xi-ode}, we notice that
	\begin{align*}
		R_{2}(t)= & \int_{r}^{t}b_{i}'(\Theta_{i,*}^{r,\omega_{i}}(s;s))R_{2}(s)\md s                                                                                                                                                                                               \\
		          & +  \int_{r}^{t}(b_{i}'(\Theta_{i,*}^{r,\omega_{i}+\eta}(s;s))-b_{i}'(\Theta_{i,*}^{r,\omega_{i}}(s;s))) \Delta \Gamma(s)\md s                                                                                                                                   \\
		          & +\int_{r}^{t} \lb[b_{i}'(\Theta_{i,*}^{r,\omega_{i}+\eta}(s;s))-b_{i}'(\Theta_{i,*}^{r,\omega_{i}}(s;s))-  b_{i}''(\Theta_{i,*}^{r,\omega_{i}}(s;s))\la \eta , \Gamma_{i,*}^{r,\omega_{i}}(s)\ra\rb] \la \tilde{\eta}, \Gamma_{i,*}^{r,\omega_{i}}(s)\ra \md s.
	\end{align*}
	By Gronwall inequality, we deduce
	\begin{align*}
		\sup_{t\in[r,T]}|R_{2}(t)| & \lesssim    \int_{r}^{T}|b_{i}'(\Theta_{i,*}^{r,\omega_{i}+\eta}(s;s))-b_{i}'(\Theta_{i,*}^{r,\omega_{i}}(s;s))||\Delta \Gamma(s)| \md s                                                                                                                   \\
		                           & + \int_{r}^{T} |b_{i}'(\Theta_{i,*}^{r,\omega_{i}+\eta}(s;s))-b_{i}'(\Theta_{i,*}^{r,\omega_{i}}(s;s))- b_{i}''(\Theta_{i,*}^{r,\omega_{i}}(s;s))\la \eta , \Gamma_{i,*}^{r,\omega_{i}}(s)\ra|| \la \tilde{\eta}, \Gamma_{i,*}^{r,\omega_{i}}(s)\ra |\md s \\
		                           & := I_{1}+I_{2}.
	\end{align*}
	By Proposition \ref{prop-theta} and Proposition \ref{prop-gamma},  we notice
	\(
	I_{1}\lesssim \|\tilde{\eta}\|_{\infty}\|\eta\|_{\infty}^{2}.
	\)
	For \(I_{2}\),  we  plug in the estimates from Proposition \ref{prop-1} and  obtain
	\begin{align*}
		I_{2} & \lesssim \|\tilde{\eta}\|_{\infty} \sup_{t\in [r,T]} |R_{1}(t)|+  \|\tilde{\eta}\|_{\infty}\int_{r}^{t} |b_{i}'(\Theta_{i,*}^{r,\omega_{i}+\eta}(s;s))-b_{i}'(\Theta_{i,*}^{r,\omega_{i}}(s;s))- b_{i}''(\Theta_{i,*}^{r,\omega_{i}}(s;s))\Delta \Theta(s)|\md s \\
		      & \lesssim   \|\tilde{\eta}\|_{\infty}\|\eta\|_{\infty}^{2} +\|\tilde{\eta}\|_{\infty} \int_{r}^{T} \int_{0}^{1}|b_{i}''(\Theta_{i,*}^{r,\omega_{i}}(s;s) +\lambda \Delta \Theta(s))-b_{i}''(\Theta_{i,*}^{r,\omega_{i}})||\Delta \Theta (s)|\md \lambda\md s.
	\end{align*}
	Notice that \(b_{i}''\) is bounded with a module of continuity \(\varrho_{i}\),and  \(\sup_{t\in[r,T]}|\Delta\Theta(t)|\leq C \|\eta\|_{\infty}\).
	By Lebesgue dominated convergence theorem, we show that \(\sup_{t\in[r,T]}R_{2}(t)\leq C\|\tilde{\eta}\|_{\infty}\|\eta\|_{\infty} \varrho_{i}(\|\eta\|_{\infty})\).
\end{proof}

Let \(c\in C^{3}(\bbR^{2};\bbR)\) be a general cost with derivatives growing at most linearly.
We consider
\begin{equation*}
	u(r,\omega_{1}, \omega_{2}):= \E\lb[\int_{r}^{T}  c(\Theta_{1,*}^{r,\omega_{1}}(t;t),\Theta_{2,*}^{r,\omega_{2}}(t;t))\md t \rb].
\end{equation*}
\begin{rmk}
	\label{rmk-c}
	For example, we can take \(c(x,y)=|f_{1}(x)-f_{2}(y)|^{2}\), where \(f_{i}\) has  bounded first, second, and third derivatives.
\end{rmk}
\begin{prop}
	\label{prop-frechet}
	Under Assumptions \ref{asmp-fsde} and \ref{asmp-k}, we have
	\(u\) is twice  Fr\'echet differentiable and weakly continuous.
	In particular, for \(i,j=1,2\),
	\begin{align}
		\partial_{\omega_{i}}u(r,\omega_{1},\omega_{2})               & =\E\lb[\int_{r}^{T}\partial_{i}c(\Theta_{1,*}^{r,\omega_{1}   }(t;t),\Theta_{2,*}^{r,\omega_{2}}(t;t))\Gamma_{i,*}^{r,\omega_{i}}(t)\md t\rb], \label{eqn-u1}                                     \\
		\partial^{2}_{\omega_{i}\omega_{j}}u(r,\omega_{1},\omega_{2}) & =\E\lb[\int_{r}^{T} \partial^{2}_{ij}c(\Theta_{1,*}^{r,\omega_{1}}(t;t),\Theta_{2,*}^{r,\omega_{2}}(t;t) )\Gamma_{i,*}^{r,\omega_{i}}(t)\otimes \Gamma_{j,*}^{r,\omega_{j}}(t)\md t\rb] \nonumber \\
		                                                              & \quad +\delta_{i,j}\E\lb[\int_{r}^{T}\partial_{i}c(\Theta_{1,*}^{r,\omega_{1}   }(t;t),\Theta_{2,*}^{r,\omega_{2}}(t;t))\Xi_{i,*}^{r,\omega_{i}}(t)\md t\rb],\label{eqn-u2}
	\end{align}
	where \(\delta_{i,j}\) is the Kronecker symbol.
\end{prop}
\begin{proof}
	The  linear growth of \(\partial^{2}_{ij}c\) yields
	\begin{align*}
		\lb|c(\tilde{\theta}_{1},\tilde{\theta}_{2})-c(\theta_{1},\theta_{2})- \sum_{i=1,2}\partial_{i}c(\theta_{1},\theta_{2})(\tilde \theta_{i}-\theta_{i})\rb|\leq C (1+\sum_{i=1,2}(|\tilde{\theta}_{i}|+|\theta_{i}| ))\sum_{i=1,2}(\tilde{\theta}_{i}-\theta_{i})^{2}.
	\end{align*}
	Plugging \(\tilde{\theta}_{i}=\Theta_{i,*}^{r,\omega_{i}+\eta_{i}}(t;t)\) and \(\theta_{i}=\Theta_{i,*}^{r,\omega_{i}}(t;t)\) into the above estimates, and by Proposition \ref{prop-1}, we deduce
	\begin{align*}
		u(r,\omega_{1}+\eta_{1},\omega_{2}+\eta_{2})- u(r,\omega_{1},\omega_{2}) & =\sum_{i=1,2}\E\lb[\int_{r}^{T}\partial_{i}c(\Theta_{1,*}^{r,\omega_{1}   }(t;t),\Theta_{2,*}^{r,\omega_{2}}(t;t))\la \eta_{i},\Gamma_{i,*}^{r,\omega_{i}}(t)\ra\md t\rb] \\
		                                                                         & \qquad + o(\|\eta_{1}\|_{\infty}+\|\eta_{2}\|_{\infty}).
	\end{align*}
	Therefore, \(u\) is Fr\'echet differentiable, and \eqref{eqn-u1} is verified.
	To show \eqref{eqn-u2}, we only need to notice that \(\partial_{i}c,\partial^{3}_{ijk}c\) has a linear growth and \(\sup_{t\in[r,T]} \la \tilde{\eta}_{i}, \Gamma_{i,*}^{r,\omega_{i}}(t)\ra \leq C \|\tilde{\eta}_{i}\|_{\infty}\).
	By Proposition \ref{prop-2} and similar arguments as above, we deduce
	\begin{align*}
		 & \quad\la \tilde{\eta}, \partial_{\omega_{1}}u(r,\omega_{1}+\eta_{1},\omega_{2}+\eta_{2})-\partial_{\omega_{1}}u(r, \omega_{1},\omega_{2})\ra                                                                                       \\
		 & =  \E\lb[ \int_{r}^{T} \partial^{2}_{12}c(\Theta_{1,*}^{r,\omega_{1}}(t;t),\Theta_{2,*}^{r,\omega_{2}}(t;t))\la \tilde{\eta}, \Gamma_{1,*}^{r,\omega_{1}}(t)\ra\la \eta_{2}, \Gamma_{2,*}^{r,\omega_{2}}(t)\ra\md t\rb]            \\
		 & \quad +\E\lb[\int_{r}^{T} \partial^{2}_{11}c(\Theta_{1,*}^{r,\omega_{1}}(t;t),\Theta_{2,*}^{r,\omega_{2}}(t;t) )\la \tilde{\eta},\Gamma_{1,*}^{r,\omega_{i}}(t)\ra \la \eta_{1},\Gamma_{1,*}^{r,\omega_{i}}(t)\ra\md t\rb]         \\
		 & \quad +\E\lb[\int_{r}^{T}\partial_{1}c(\Theta_{1,*}^{r,\omega_{1}   }(t;t),\Theta_{2,*}^{r,\omega_{2}}(t;t))\la (\tilde{\eta},\eta_{1}),\Xi_{1,*}^{r,\omega_{1}}(t)\ra\md t\rb]+ o( \|\eta_{1}\|_{\infty}+ \|\eta_{2}\|_{\infty}).
	\end{align*}
	Therefore, \(u\) is twice  Fr\'echet differentiable and weakly continuous with derivatives given in \eqref{eqn-u1} and \eqref{eqn-u2}.
\end{proof}

\begin{proof}[Proof of Lemma \ref{lem-reg}]
	We first show that \(u\) satisfies all conditions in Lemma \ref{lem-ito}.
	We recall the regularity condition here.
	For  any  \(\eta , \tilde{\eta}\in C([0,T];\bbR)\), it holds that
	\begin{enumerate}[label=(\roman*)]
		\item for any \(\omega_{1},\omega_{2}\in C([0,T];\bbR)\), \begin{equation}
			      \label{eqn-reg1}
			      \begin{aligned}
				      |\la \eta, \partial_{\omega_{i}}u(r,\omega_{1},\omega_{2})\ra|                              & \leq C (1+\|\omega_{1}\|_{\infty}+\|\omega_{2}\|_{\infty}) \|\eta\|_{\infty}                           \\
				      |\la (\eta,\tilde{\eta}), \partial^{2}_{\omega_{i}\omega_{j}}u(r,\omega_{1},\omega_{2})\ra| & \leq C (1+\|\omega_{1}\|_{\infty}+\|\omega_{2}\|_{\infty}) \|\eta\|_{\infty}\|\tilde{\eta}\|_{\infty};
			      \end{aligned}
		      \end{equation}
		\item for any other \(\omega_{1}',\omega_{2}'\in C([0,T];\bbR)\), there exists a modulus of continuity \(\rho\) such that
		      \begin{equation}
			      \label{eqn-reg2}
			      \begin{aligned}
				       & |\la (\eta,\tilde{\eta}), \partial^{2}_{\omega_{i}\omega_{j}}u(r,\omega_{1},\omega_{2})-\partial^{2}_{\omega_{i}\omega_{j}}u(r,\omega_{1}',\omega_{2}')\ra|                                     \\
				       & \hspace{3cm}\leq C (1+\|\omega_{1}\|_{\infty}+\|\omega_{2}\|_{\infty}) \|\eta\|_{\infty}\|\tilde{\eta}\|_{\infty}\rho(\|\omega_{1}-\omega_{1}'\|_{\infty}+\|\omega_{2}-\omega_{2}'\|_{\infty}).
			      \end{aligned}
		      \end{equation}
	\end{enumerate}
	We first verify \eqref{eqn-reg1}.
	By Propositions \ref{prop-theta} and \ref{prop-gamma}, we have
	\begin{equation*}
		\sup_{t\in[r,T]}\E[|\Theta_{i,*}^{r,\omega_{i}}(t;t)|]\leq C(1+\|\omega_{i}\|_{\infty}) \text{ and }|\la \eta,\Gamma_{i,*}^{r,\omega_{i}}(t)\ra| \leq C\|\eta\|_{\infty}.
	\end{equation*}
	Plugging the above into \eqref{eqn-u1}, we derive
	\begin{align*}
		|\la \eta, \partial_{\omega_{i}}u(r,\omega_{1},\omega_{2})\ra| & \leq C \|\eta\|_{\infty} \E\lb[1+ |\Theta_{1,*}^{r,\omega_{1}   }(t;t)|+|\Theta_{2,*}^{r,\omega_{2}}(t;t)|\rb] \\
		                                                               & \leq C (1+\|\omega_{1}\|_{\infty}+\|\omega_{2}\|_{\infty}) \|\eta\|_{\infty}.
	\end{align*}
	For the second derivative, we notice
	\begin{align*}
		|\la(\eta,\tilde{\eta}), \Gamma_{i,*}^{r,\omega_{i}}(t)\otimes \Gamma_{j,*}^{r,\omega_{j}}(t)\ra|=|\la \eta,\Gamma_{i,*}^{r,\omega_{i}}(t)\ra\la \tilde{\eta},\Gamma_{j,*}^{r,\omega_{j}}(t) \ra|\leq C \|\eta\|_{\infty}\|\tilde{\eta}\|_{\infty}.
	\end{align*}
	Moreover, from \eqref{eqn-xi} and the boundedness of \(b_{i}',b_{i}''\), we deduce \[|\la (\eta,\tilde{\eta}), \Xi_{i,*}^{r,\omega_{i}}(t)\ra|\leq  C \int_{r}^{t}|\la(\eta,\tilde{\eta}), \Gamma_{i,*}^{r,\omega_{i}}(s)\otimes \Gamma_{i,*}^{r,\omega_{i}}(s)\ra|\md s\leq C \|\eta\|_{\infty}\|\tilde{\eta}\|_{\infty}. \]
	Therefore, by Proposition \ref{prop-frechet} and the linear growth of \(\partial_{i}c,\partial^{2}_{ij}c\),  we derive
	\begin{align*}
		|\la (\eta,\tilde{\eta}), \partial^{2}_{\omega_{i}\omega_{j}}u(r,\omega_{1},\omega_{2})\ra| & \leq C (1+\|\omega_{1}\|_{\infty}+\|\omega_{2}\|_{\infty}) \|\eta\|_{\infty}\|\tilde{\eta}\|_{\infty}.
	\end{align*}

	Now, we start to verify \eqref{eqn-reg2}.
	Since \(\partial^{2}_{ij}c\) has a linear growth, we have
	\begin{align*}
		 & \quad\lb|\partial_{i}c(\Theta_{1,*}^{r,\omega_{1}   }(t;t),\Theta_{2,*}^{r,\omega_{2}}(t;t))-\partial_{i}c(\Theta_{1,*}^{r,\omega_{1}'   }(t;t),\Theta_{2,*}^{r,\omega_{2}'}(t;t))\rb| \\
		 & \leq  C(1+ |\Theta_{1,*}^{r,\omega_{1}   }(t;t)|+|\Theta_{2,*}^{r,\omega_{2}}(t;t)|+|\Theta_{1,*}^{r,\omega_{1}'   }(t;t)|+|\Theta_{2,*}^{r,\omega_{2}'}(t;t)|)                        \\
		 & \hspace{3cm}\times(|\Theta_{1,*}^{r,\omega_{1}   }(t;t)-\Theta_{1,*}^{r,\omega_{1}'}(t;t)|+|\Theta_{2,*}^{r,\omega_{2}   }(t;t)-\Theta_{2,*}^{r,\omega_{2}'}(t;t)|)                    \\
		 & \leq C(1+ |\Theta_{1,*}^{r,\omega_{1}   }(t;t)|+|\Theta_{2,*}^{r,\omega_{2}}(t;t)|+ \|\omega_{1}-\omega_{1}'\|_{\infty}+\|\omega_{2}-\omega_{2}'\|_{\infty})                           \\
		 & \hspace{3cm} \times(\|\omega_{1}-\omega_{1}'\|_{\infty}+\|\omega_{2}-\omega_{2}'\|_{\infty}).
	\end{align*}
	Similarly, as \(\partial^{3}_{ijk}c\) has a linear growth, we have
	\begin{align*}
		 & \quad\lb|\partial^{2}_{ij}c(\Theta_{1,*}^{r,\omega_{1}   }(t;t),\Theta_{2,*}^{r,\omega_{2}}(t;t))-\partial^{2}_{ij}c(\Theta_{1,*}^{r,\omega_{1}'   }(t;t),\Theta_{2,*}^{r,\omega_{2}'}(t;t))\rb| \\
		 & \leq C(1+ |\Theta_{1,*}^{r,\omega_{1}   }(t;t)|+|\Theta_{2,*}^{r,\omega_{2}}(t;t)|+ \|\omega_{1}-\omega_{1}'\|_{\infty}+\|\omega_{2}-\omega_{2}'\|_{\infty})                                     \\
		 & \hspace{3cm} \times(\|\omega_{1}-\omega_{1}'\|_{\infty}+\|\omega_{2}-\omega_{2}'\|_{\infty}).
	\end{align*}
	By Proposition \ref{prop-gamma}, we have
	\begin{align*}
		 & \quad    |\la (\eta,\tilde{\eta}), \Gamma_{i,*}^{r,\omega_{i}}(t)\otimes \Gamma_{j,*}^{r,\omega_{j}}(t)- \Gamma_{i,*}^{r,\omega'_{i}}(t)\otimes \Gamma_{j,*}^{r,\omega'_{j}}(t)\ra|                                                                                            \\
		 & \leq C |\la \eta, \Gamma_{i,*}^{r,\omega_{i}}(t) -\Gamma_{i,*}^{r,\omega'_{i}}(t)\ra \la \tilde{\eta},\Gamma_{j,*}^{r,\omega_{j}}(t)\ra|+ C |\la \eta, \Gamma_{i,*}^{r,\omega'_{i}}(t) \ra \la \tilde{\eta},\Gamma_{j,*}^{r,\omega_{j}}(t)-\Gamma_{j,*}^{r,\omega'_{j}}(t)\ra| \\
		 & \leq C\|\eta\|_{\infty}\|\tilde{\eta}\|_{\infty}(\|\omega_{i}-\omega'_{i}\|_{\infty}+\|\omega_{j}-\omega'_{j}\|_{\infty}).
	\end{align*}
	Plugging the above estimates into \eqref{eqn-xi}, we derive
	\begin{equation*}
		|\la (\eta,\tilde{\eta}), \Xi_{i,*}^{r,\omega_{i}}(t)-\Xi_{i,*}^{r,\omega'_{i}}(t)\ra|\leq C \|\eta\|_{\infty}\|\tilde{\eta}\|_{\infty}\varrho_{i}(\|\omega_{i}-\omega'_{i}\|_{\infty}),
	\end{equation*}
	where \(\varrho_{i}\) is the modulus of continuity of \(b_{i}''\).
	Combining the above estimates, we conclude \eqref{eqn-reg2}.

	Now, we show that \(\partial_{t}u\) exists and is continuous.
	By the Markov property of \((\Theta_{1,*}^{r,\omega_{1}},\Theta_{2,*}^{r,\omega_{2}})\), we have
	\begin{equation}
		\label{eqn-dpp}
		u(r,\omega_{1},\omega_{2})=\E\lb[\int_{r}^{r+\delta}c\lb(\Theta_{1,*}^{r,\omega_{1}}(t;t),\Theta_{2,*}^{r,\omega_{2}}(t;t)\rb) \md t + u(r+\delta, \Theta_{1,*}^{r,\omega_{1}}(r+\delta;\cdot),\Theta_{2,*}^{r,\omega_{2}}(r+\delta;\cdot))\rb].
	\end{equation}
	Since we have verified \eqref{eqn-reg1} and \eqref{eqn-reg2},  applying  It\^o formula we obtain \begin{align*}
		 & \quad u(r+\delta, \Theta_{1,*}^{r,\omega_{1}}(r+\delta;\cdot),\Theta_{2,*}^{r,\omega_{2}}(r+\delta;\cdot))-u(r+\delta,\omega_{1},\omega_{2})                                                                                                                                                            \\
		 & = \int_{r}^{r+\delta} (\cL_{1}+\cL_{2})u(s,\Theta_{1,*}^{r,\omega_{1}}(s;\cdot),\Theta_{2,*}^{r,\omega_{2}}(s;\cdot))+\la (k_{1}(\cdot,s),k_{2}(\cdot,s)), \partial_{\omega_{1}\omega_{2}}^{2}u(s, \Theta_{1,*}^{r,\omega_{1}}(s;\cdot),\Theta_{2,*}^{r,\omega_{2}}(s;\cdot))\ra \md s                  \\
		 & \quad + \int_{r}^{r+\delta}\lb[\la k_{1}(\cdot,s),\partial_{\omega_{1}}u(s, \Theta_{1,*}^{r,\omega_{1}}(s;\cdot),\Theta_{2,*}^{r,\omega_{2}}(s;\cdot))\ra +\la k_{2}(\cdot,s),\partial_{\omega_{2}}u(s, \Theta_{1,*}^{r,\omega_{1}}(s;\cdot),\Theta_{2,*}^{r,\omega_{2}}(s;\cdot))\ra \rb]\md B_{1}(s).
	\end{align*}
	Plug the above identity into \eqref{eqn-dpp} and divide both sides by \(\delta\).
	Let \(\delta\) go to 0, and we deduce \(u\) satisfies
	\begin{equation}
		(\partial_{t}+\cL_{1}+\cL_{2})u(r,\omega_{1},\omega_{2})+\la (k_{1}(\cdot,r),k_{2}(\cdot,r)), \partial^{2}_{\omega_{1}\omega_{2}}u(r,\omega_{1},\omega_{2})(\cdot)\ra = -c(\omega_{1}(r),\omega_{2}(r)).
	\end{equation}
	This gives the continuity of \(\partial_{t}u\).
	We conclude the proof by noticing \(u=V_{*}\) if we take \(c(x,y)=|x-y|^{2}\).
\end{proof}

\section*{Acknowledgement}
The authors would like to thank Professor Jan Ob\l\'{o}j, Professor Rama Cont, Professor Julio Backhoff, and Vlad Tuchilus for their insightful comments and suggestions.
Y. Jiang acknowledges the support by the EPSRC Centre for Doctoral Training in Mathematics of Random Systems: Analysis, Modelling, and Simulation (EP/S023925/1).

\bibliography{mybib}
\end{document}